\documentclass[11pt]{amsart}

\usepackage[text={455pt,670pt},centering]{geometry}
\usepackage{color}
\usepackage{amssymb}
\usepackage{graphicx}
\usepackage{tikz}
\usepackage{hyperref}
\usepackage{dsfont}

\newtheorem{theorem}{Theorem}
\newtheorem{proposition}[theorem]{Proposition}
\newtheorem{lemma}[theorem]{Lemma}

\theoremstyle{definition}
\newtheorem{remark}[theorem]{Remark}

\numberwithin{equation}{section}
\numberwithin{figure}{section}
\numberwithin{theorem}{section}

\newcommand{\N}{\mathbb{N}}
\newcommand{\R}{\mathbb{R}}

\newcommand{\abs}[1]{\left| #1 \right|}

\renewcommand{\tilde}{\widetilde}

\newcommand{\eps}{\varepsilon}
\newcommand{\pa}{\partial}

\definecolor{refkey}{gray}{0.75}
\colorlet{labelkey}{blue} 
\begin{document}

\title[Two-species kinetic model of chemotaxis]
{Existence and diffusive limit of a two-species kinetic model of chemotaxis}

\begin{abstract}
In this paper, we propose a kinetic model describing the collective motion
by chemotaxis of two species in interaction emitting the same chemoattractant.
Such model can be seen as a generalisation to several species 
of the Othmer-Dunbar-Alt model which takes into account the run-and-tumble 
process of bacteria.
Existence of weak solutions for this two-species kinetic model is studied and 
the convergence of its diffusive limit towards a macroscopic model of 
Keller-Segel type is analysed.

\end{abstract}

\author[L. Almeida]{Luis Almeida}
\address{
CNRS, UMR 7598, Laboratoire Jacques-Louis Lions, F-75005, Paris, France  \\
Sorbonne Universit\'es, UPMC Univ Paris 06, UMR 7598, Laboratoire Jacques-Louis Lions, F-75005, Paris, France \\
INRIA-Paris-Rocquencourt, EPC MAMBA, Domaine de Voluceau, BP105, 78153 Le Chesnay Cedex}
\email{luis@ann.jussieu.fr}

\author[C. Emako-Kazianou]{Casimir Emako-Kazianou}
\address{Sorbonne Universit\'es, UPMC Univ Paris 06, UMR 7598, Laboratoire Jacques-Louis Lions, F-75005, Paris, France\\
CNRS, UMR 7598, Laboratoire Jacques-Louis Lions, F-75005, Paris, France \\
INRIA-Paris-Rocquencourt, EPC MAMBA, Domaine de Voluceau, BP105, 78153 Le Chesnay Cedex}
\email{emako@ann.jussieu.fr}

\author[N. Vauchelet]{Nicolas Vauchelet}
\address{Sorbonne Universit\'es, UPMC Univ Paris 06, UMR 7598, Laboratoire Jacques-Louis Lions, F-75005, Paris, France \\
CNRS, UMR 7598, Laboratoire Jacques-Louis Lions, F-75005, Paris, France \\
INRIA-Paris-Rocquencourt, EPC MAMBA, Domaine de Voluceau, BP105, 78153 Le Chesnay Cedex}
\email{vauchelet@ann.jussieu.fr}

\date{\today}

\maketitle

\section{Introduction}
Chemotaxis is the biological mechanism by which organisms sense their 
environment and react to chemical stimuli. 
It induces a motion towards the attractant (positive chemotaxis) or away from 
the repellent (negative chemotaxis).
One consequence of positive chemotaxis is the formation of patterns 
and aggregates as observed in \cite{BudBerg,Park,Hofer,KSegel} 
for motile bacteria Escherichia coli or Dictyostelium discoideum mold.
Many mathematical models have been proposed to explain this aggregation
phenomena. Among them, we can distinguish between microscopic and 
macroscopic models depending on the level of description.

In \cite{OthmerAlt}, Othmer, Dunbar and Alt choose the microscopic setting in which cells are represented by their velocity distribution $f(x,v,t)$ and the chemical attractant (chemoattractant) by its concentration $S(x,t)$.
The dynamics of $f$ is given by
\begin{equation}\label{ODA_kinetic}
 \partial_t f+v\cdot\nabla_x f=\int_{V}\big(T[S](x,v,v',t)f(x,v',t)-T[S](x,v',v,t)f(x,v,t)\big)\,dv',
\end{equation}
with $V$ a bounded domain of $\mathbb{R}^d$.
This kinetic equation points out the {\it run-and-tumble}
process which characterises the individual motion of cells (see \cite{Berg}).
The left-hand side is associated to the run phase during which cells move 
in a straight line at a constant speed $v$. 
The right-hand side includes loss and gain terms resulting from the 
reorientation phase (tumble).
Cells reorient from $v$ to $v'$ with the probability per unit of time $T[S](x,v',v,t)/\int_V T[S](x,v',v,t)\,dv'$.
This justifies why $T[S]$ is called the tumbling rate or kernel. 
Many choices of $T[S]$ are possible.
Since cells are able to respond to temporal changes of the gradient 
of the chemical substance $S$ along their pathways, we opt for the form of $T[S]$ proposed in \cite{DolakSch}~:
\begin{equation}\label{T[S]_def}
  \forall x\in \mathbb{R}^d,\,v,v'\in V,\, t>0, \quad  T[S](x,v,v',t):=\phi(\partial_t S+v'\cdot \nabla_x S),
\end{equation}
where $\phi$ is decreasing in order to take into account the preference 
for favourable regions. This model has shown to be efficient to 
describe the traveling pulse behaviour of bacteria observed 
experimentally in \cite{Saragosti}.
Finally, the chemoattractant is emitted by the cells themselves, diffuses 
into the medium and is naturally degraded. Then, the chemoattracttant
concentration $S$ appearing in \eqref{ODA_kinetic} solves the 
following reaction-diffusion equation~:
\begin{equation}\label{eq_S}
 \delta \partial_t S-\Delta S +S=\int_{V}f(x,v,t)\,dv:=\rho(x,t),\quad \delta=0,1.
\end{equation}


At a macroscopic level, the dynamics of cells is described by their
density $\rho(t,x)$.
The well-known Keller-Segel \cite{Keller} system has been widely used to
describe aggregation by chemotaxis. This model describes the dynamics
of cells thanks to a parabolic equation with an oriented drift 
depending on the spatial gradient of the chemoattractant~:
\begin{equation}\label{Keller_Segel}
\left\{
\begin{aligned}
 &\partial_t \rho=\nabla\cdot\left(D\nabla \rho-\chi \rho \nabla S \right),\\
 &\delta \partial_t S-\Delta S +S =\rho,\quad \delta=0,1,
\end{aligned}
\right.
\end{equation}
where $D$ and $\chi$ are positive constants called the diffusivity and 
the chemosensivity of the species to the chemoattractant.

In the mathematical litterature both elliptic ($\delta =0$) and 
parabolic ($\delta =1$) cases are encountered.
Although the point of view of microscopic and macroscopic models is different,
it has been proved that the Keller-Segel model \eqref{Keller_Segel} 
can be derived as the diffusion limit of the Othmer-Dunbar-Alt model
\eqref{ODA_kinetic}--\eqref{eq_S} (see \cite{Bellomo,Bellouquid,CMPS,HKS,VCJS,BP}).
The hyperbolic limit can also be considered \cite{DolakSch,FilbetP,jamesnv}
leading to the same kind of macroscopic model with small diffusion.
As a consequence, coefficients $D$ and $\chi$ of \eqref{Keller_Segel} 
depend on microscopic parameters which can be measured. 
This allows one to fit the model with experimental data as done in \cite{VCJS}. 
Other advantages of \eqref{Keller_Segel} are understanding of 
collective effects emerging from individual behaviours and its simple simulation 
compared to \eqref{ODA_kinetic}--\eqref{eq_S}. 
However, the microscopic approach provides a general framework of 
chemotaxis models which encompasses macroscopic models including hyperbolic models obtained by 
a momentum method from kinetic models (see e.g. \cite{Greenberg,Hillen,HillenStevens,Ribot}).


In this work, we focus on the modelling of the chemotactic behaviour
of two-interacting species.
Existing two-species models (see e.g. \cite{Horstmann,Wolansky,Fasano,Marco}) 
concern the macroscopic scale. For instance, the following Keller-Segel 
two-species model is considered~:
\begin{equation}\label{Keller_Segel_2}
 \left\{
 \begin{aligned}
  &\partial_t \rho_1=\nabla \cdot\big(D_1\nabla \rho_1-\chi_1\rho_1 \nabla S\big),\\
  &\partial_t \rho_2=\nabla \cdot\big(D_2\nabla \rho_2-\chi_2 \rho_2\nabla S\big),\\
  &\delta \partial_t S-\Delta S+S=\rho_1+\rho_2,\quad \delta=0,1,
 \end{aligned}
\right.
\end{equation}
where $D_1,D_2$ and $\chi_1,\chi_2$ are the diffusivities and chemosensivities of the two species 1, 2 to the common chemoattractant $S$. Which can happen in case we consider two closely related types of cells.
Many theoretical issues arise from \eqref{Keller_Segel_2}. The question of global existence of solutions and understanding of the blow-up are addressed in \cite{Conca,CEEVK,Espejo} in the two-dimensional case. 
These results are validated by numerical simulations carried out in \cite{Kurganov}. In addition, traveling wave solutions of a two-species model like \eqref{Keller_Segel_2} are studied in \cite{Lin}. 

In this paper, we address the question of the derivation of such macroscopic 
model from a kinetic point of view.
We propose the following  microscopic model in which the dynamics of the distribution function $f_i(x,v,t)$ for the $i$-th species,
$i=1,2$, is governed by the two following kinetic equations~: 
\begin{equation}\label{kinetics}
 \left\{
 \begin{aligned}
  & \partial_t f_i  +v\cdot \nabla_x f_i =\int_{V}\left(T_i[S](x,v,v',t)f_i(x,v',t)-T_i[S](x,v',v,t)f_i(x,v,t)\right)dv',\\
  & f_i(x,v,t=0)=f_i^{ini}(x,v),\quad \text{ for }i=1,2.\\
 \end{aligned}
 \right.
\end{equation}
The position $x \in \mathbb{R}^d$, velocity $v\in V$ (where $V$ is a bounded set of $\mathbb{R}^d$) and time $t\geq 0$.
As previously, the tumbling rate $T_i[S]$ takes into account temporal changes
of the chemoattractant concentration along the path of cells and reads~:
\begin{equation}\label{tumbling_kernel}
  \forall x\in \mathbb{R}^d,\, v,v'\in V,\, t>0,\quad T_i[S](x,v,v',t):=\phi_i(\partial_t S+v'\cdot \nabla_x S),\quad \text {for }i=1,2,
\end{equation}
where $\phi_i$ is a decreasing function.
We consider the case where species 1 and 2 involved in \eqref{kinetics}
emit the same attracting chemical substance $S$, whose dynamics is given 
by the parabolic ($\delta =1$) or elliptic ($\delta =0$) system~:
\begin{equation}\label{eq_S2}
\left\{
\begin{aligned}
  & \delta \partial_t S-\Delta S +S =\int_{V} f_1(x,v,t) dv +\int_{V} f_2(x,v,t) dv:=\rho_1(x,t)+\rho_2(x,t) , \quad \delta=0,1,\\
   & S(x,t=0)=0, \qquad \qquad \mbox{ if } \delta =1.
\end{aligned}
\right.
\end{equation}
We determine the drift-diffusion limit of \eqref{kinetics}--\eqref{eq_S2} by performing a diffusive scaling of space and time $\tilde{x}=\varepsilon x,\,\tilde{t}=\varepsilon^2 t$. After dropping the tilde, 
system \eqref{kinetics} now reads
\begin{equation}\label{kinetics_eps}
 \left\{
 \begin{aligned}
  & \varepsilon^2\partial_t f_i^{\varepsilon}  +\varepsilon v\cdot \nabla_x f_i^{\varepsilon} =-\mathcal{T}_i^{\varepsilon}[S^{\varepsilon}](f_i^{\varepsilon}),\\
 &f_i^{\varepsilon}(x,v,t=0)=f_i^{ini}(x,v),\quad \text{ for } \quad i=1,2,
 \end{aligned}
 \right.
\end{equation}
 with
  \begin{equation*}
  \mathcal{T}_i^{\varepsilon}[S](f):=\int_{V}\left(T_i^{\varepsilon}[S](x,v',v,t)f(x,v,t)-T_i^{\varepsilon}[S](x,v,v',t)f(x,v',t)\right)dv',
 \end{equation*}
where we consider as above
 \begin{equation}\label{def_Teps}
  \forall x\in \mathbb{R}^d,\, v,v'\in V,\, t>0,\quad 
T_i^{\varepsilon} [S](x,v,v',t):=\phi_i^{\varepsilon}(\varepsilon \partial_t S+v'\cdot \nabla_x S),\quad \text {for }i=1,2.\\      
\end{equation}
The difference of scale between terms $\partial_t S$ and $v\cdot \nabla_x S$ comes from the scaling between space and time.
The equation for $S^{\varepsilon}$ is unchanged and stated now as
\begin{equation}\label{eq_S_eps}
\left\{
\begin{aligned}
 & \delta \partial_t S^{\varepsilon}-\Delta S^{\varepsilon}+S^{\varepsilon}=\rho_1^{\varepsilon}+\rho_2^{\varepsilon}, \qquad \qquad  \delta = 0,1,\\
 & S^{\varepsilon}(x,t=0)=0, \qquad \qquad \mbox{if } \delta =1.
\end{aligned}
\right.
\end{equation}

We first prove the global-in-time existence of solution $(f_1,f_2,S)$  
of \eqref{kinetics}--\eqref{eq_S2}.
Then, we prove convergence when $\eps\to 0$ of solutions to 
\eqref{kinetics_eps}--\eqref{eq_S_eps} towards a macroscopic 
model of Keller-Segel type.
The proof relies on uniform estimates on 
$f_1^{\varepsilon},f_2^{\varepsilon},S^{\varepsilon}$
which allow us to use the Aubin-Lions-Simon compactness Lemma \cite{JS}.
We only focus on the case of bounded tumbling kernel $T$ for which
no blow-up of solutions in finite time is expected both at the microscopic
and macroscopic levels. 
This non blow-up has been proved in the one species case in \cite{Chertock}.

The paper is organised as follows. The following Section presents the two main results~: 
the global-in-time existence theorem of solutions to 
\eqref{kinetics_eps}--\eqref{eq_S_eps} for fixed $\eps>0$ 
and the convergence as $\eps\to 0$ of this solution towards a macroscopic model,
i.e. the drift-diffusion limit of \eqref{kinetics_eps}--\eqref{eq_S_eps}.
We also formally derive in this Section the equation verified by the limit. 
By a fixed-point argument, global existence of solutions to the kinetic model
is proved in Section~\ref{global_existence}.
Section \ref{drift_diffusion} is devoted to the proof 
of the drift-diffusion limit stated in Theorem~\ref{limit}.
Finally, we explain in the appendix the non blow-up of solutions of the
derived macroscopic model.

\section{Main results}\label{Main results}

Before stating our main results, we introduce some notations.
We denote $\Omega:=\mathbb{R}^d \times V$ where $V$ is a bounded 
and symmetric domain of $\R^d$, i.e. if $v\in V$ then $-v \in V$.
For $1\leq p,q\leq \infty$, $k \in \mathbb{N}^*$ and $\tau>0$, we define
\begin{itemize}
 \item $L^p_{+}(\mathbb{R}^d)$ the set of nonnegative functions in $L^p(\mathbb{R}^d)$.
 \item $W^{k,q}(\mathbb{R}^d)$ the space of functions $u$ such that for any 
$\gamma\in \N^d$ with $|\gamma|\leq k$, $D^{\gamma}u\in L^q(\mathbb{R}^d)$.
 \item $C^{0,\alpha}(\mathbb{R}^d)$, for $0<\alpha \leq 1$,  the space of 
H\"older continuous functions with exponent $\alpha$. It is equipped with the norm
 \begin{equation*}
  \left\|u\right\|_{C^{0,\alpha}(\mathbb{R}^d)}:=\left\|u\right\|_{L^{\infty}(\mathbb{R}^d)}+\sup_{x \neq y} \frac{\abs{u(x)-u(y)}}{\abs{x-y}^{\alpha}}.
 \end{equation*}
\item $C^{k,\alpha}(\mathbb{R}^d)$, for $0<\alpha\leq 1$, the space of functions whose derivatives up to the $k$-th order are H\"older continuous with exponent $\alpha$.
\item $L^p((0,\tau),B)$, for any Banach space $B$ on $\R^d$, the space 
of functions $u$ such that for a.e. $t\in (0,\tau)$, $u(\cdot,t)\in B$ and 
$t\mapsto \|u(\cdot,t)\|_B$ belongs to $L^p((0,\tau))$. It is 
endowed with the norm~:
\begin{equation*}
\|u\|_{L^p((0,\tau),B)}:=\left(\int_0^{\tau} \left\|u(\cdot,t)\right\|_B^pdt \right)^{1/p}.
\end{equation*}
\end{itemize}
Finally, we define the following abbreviations which are used throughout the paper~:
 \begin{align*}
 &f_i^{\varepsilon}:=f_i^{\varepsilon}(x,v,t),\quad f_i'^{\varepsilon}:=f_i^{\varepsilon}(x,v',t),\\
 &T_i^{\varepsilon}[S]:=T_i^{\varepsilon}[S](x,v',v,t),\quad T_i^{*,\varepsilon}[S]:=T_i^{\varepsilon}[S](x,v,v',t), \qquad \mbox{for } i=1,2.\\
 \end{align*}

\subsection{Main results}\label{theorems}

In this paper, we consider tumbling rates $T_i^{\varepsilon}[S]$ of the form \eqref{def_Teps} which meet the following requirement for $i=1,2$
\begin{enumerate}
\item[(H1)]
$ \phi_i^{\varepsilon}(z)=\psi_i\big(1+\varepsilon \theta_i(z)\big),$
with $\psi_i \in \R_+^\ast$ and $\theta_i \in C^{0,1}(\R) \cap L^\infty(\R)$ 
is nonincreasing and satisfies 
$ \left\|\theta_i\right\|_{L^{\infty}(\mathbb{R})}<1$.
\end{enumerate}

\begin{remark}
 This hypothesis is realistic since it was observed experimentally in \cite{Saragosti} that for bacteria \textit{E. Coli}, an external stimulus modifies their natural constant tumbling kernel by adding an anisotropic small term.  
\end{remark}

Note that the condition on the $L^{\infty}$-norm of $\theta_i$ ensures that $T_i^{\varepsilon}$ is positive at least for $\varepsilon$ smaller than $1$, 
which will always be the case here since we focus on the asymptotic limit $\varepsilon\to 0$. 
The positivity and the boundness of $T_i^{\varepsilon}[S]$ come from its physical meaning.

If $F$ denotes the uniform distribution on $V$
\begin{equation}\label{def_F}
 F(v):=\frac{\mathds{1}_{ v\in V}}{\abs{V}},
\end{equation}
with $\abs{V}$ the measure of the velocity set $V$.
Then, the symmetry assumption of $V$ implies that
\begin{equation}\label{momentum}
 \int_V F(v)dv=1 \quad \text{and} \quad  \int_V v F(v)dv=0.
\end{equation}

As in \cite{CMPS}, we define the symmetric and anti-symmetric parts of $\mathcal{T}_i^{\varepsilon}$ by
\begin{equation}\label{decomp_T}
 \begin{aligned}
\phi_i^{S,\varepsilon}[S]:&=\frac{T_i^{\varepsilon}[S]+T_i^{*,\varepsilon}[S]}{2}=\psi_i\left(1+\frac{\varepsilon}{2}\big(\theta_i(\varepsilon \partial_t S+v'\cdot \nabla_x S)+\theta_i(\varepsilon \partial_t S+v\cdot \nabla_x S)\big)\right),\\
\phi_i^{A,\varepsilon}[S]:&=\frac{T_i^{\varepsilon}[S]-T_i^{*,\varepsilon}[S]}{2}=\psi_i\frac{\varepsilon}{2}\big(\theta_i(\varepsilon \partial_t S+v'\cdot \nabla_x S)-\theta_i(\varepsilon \partial_t S+v\cdot \nabla_x S)\big).
\end{aligned}
\end{equation}

From Assumption (H1), $\phi_i^{S,\varepsilon}$ and $\phi_i^{A,\varepsilon}$ satisfy the following inequalities which are useful to derive uniform estimates in $\varepsilon$~: 
\begin{equation}\label{ineq_phi}
\begin{aligned}
\phi_i^{S,\varepsilon} & \geq \psi_i\big(1-\varepsilon \left\|\theta_i\right\|_{L^{\infty}(\mathbb{R})}\big),\\
\int_V \frac{(\phi_i^{A,\varepsilon})^2}{\phi_i^{S,\varepsilon}}dv'& \leq \varepsilon^2 \psi_i \abs{V} \left\|\theta_i\right\|_{L^{\infty}(\mathbb{R})}^2.
\end{aligned}
\end{equation}
The expansion of $\mathcal{T}_i^{\varepsilon}[S]$ now reads~:
 \begin{equation}\label{dev_T}
 \mathcal{T}_i^{\varepsilon}[S]=\mathcal{T}_i^{0}+\varepsilon \mathcal{T}_i^{1}[S],
\end{equation}
with 
\begin{equation}\label{def_T_ik}
\begin{aligned}
\mathcal{T}_i^{0}(f)&:=\psi_i\left(\abs{V} f-\rho\right),\\
\mathcal{T}_i^{1}[S](f)&:=\psi_i\big(\abs{V}\theta_i(\varepsilon \partial_t S+v\cdot \nabla_x S)f-\int_V\!\!\theta_i(\varepsilon \partial_t S+v'\cdot \nabla_x S) f'\,dv'\big).
\end{aligned}
\end{equation}
Our first result concerns the global-in-time existence of solution of \eqref{kinetics_eps}--\eqref{eq_S_eps}.
\begin{theorem} \label{existence}
Let $\varepsilon>0$ and assume that tumbling rates $T_1^{\varepsilon},T_2^{\varepsilon}$ are given by \eqref{def_Teps} where $\phi_1^{\varepsilon},\phi_2^{\varepsilon}$ are positive, bounded and Lipschitz continuous functions.

If the initial data $f^{ini}_1,f^{ini}_2$ are in $L^1_{+}(\Omega) \cap L^{\infty}(\Omega)$, then there exists a unique global solution of \eqref{kinetics_eps}--\eqref{eq_S_eps} such that
\begin{equation*}
\begin{aligned}
f_1^{\varepsilon},f_2^{\varepsilon} & \in L^{\infty}((0,\infty),L^1_{+} \cap \, L^{\infty}(\Omega)),\\
S^{\varepsilon} & \in L^{\infty}((0,\infty),L^p(\mathbb{R}^d)),\quad \text{for all }  1\leq p\leq \infty.
 \end{aligned}
\end{equation*}
\end{theorem}

Then we establish the diffusive limit $\varepsilon\to 0$ of these solutions.
The limiting system is the following two-species Keller-Segel type equation~:
\begin{equation}\label{limit_equation}
 \left\{
 \begin{aligned}
  &\partial_t \rho_1=\nabla \cdot \left(D_1\nabla \rho_1-\chi_1[S]\rho_1\right),\\
  &\partial_t \rho_2=\nabla \cdot \left(D_2\nabla \rho_2-\chi_2[S]\rho_2\right),\\
  &\delta \partial_t S=\Delta S-S+\rho_1+\rho_2,\quad \delta=0,1,
 \end{aligned}
 \right.
\end{equation}
where $D_i$ and $\chi_i[S]$ are given for $i=1,2$ by
\begin{equation}\label{coef_limit_eq}
  D_i=\frac{1}{\abs{V}^2\psi_i} \int_V v\otimes v \,dv ,\quad \chi_i[S]=-\int_V\!\!v\,\theta_i(v\cdot \nabla_x S^0)\frac{dv}{\abs{V}}.
\end{equation}
The intial conditions of this system are
\begin{equation*}
 \rho_{1}^{ini}=\int_V f_1^{ini} dv,\quad \rho_{2}^{ini}=\int_V f_2^{ini} dv \quad \mbox{and }  S^{ini}=0 \quad \mbox{if } \delta=1.
\end{equation*}

\begin{theorem}\label{limit}
Let (H1) hold. Assume that the initial data $f_1^{ini},f_2^{ini}$ belong to $L^1_{+}(\Omega)\cap L^{\infty}(\Omega)$.
Then, there exists a subsequence $(f_1^{\varepsilon},f_2^{\varepsilon},S^{\varepsilon})$ of solutions of \eqref{kinetics_eps}--\eqref{eq_S_eps} that converges when $\varepsilon$ tends to zero and we have
\begin{equation*}
 \begin{aligned}
  & (f_1^{\varepsilon},f_2^{\varepsilon}) \overset{\ast}{\rightharpoonup} (\rho_1^0 F,\rho_2^0 F) \quad  \text{in } L^{\infty}_{\text{loc}}((0,\infty),L^q(\Omega)),\quad 1<q<\infty\\ 
  & (S^{\varepsilon},\nabla_x S^{\varepsilon})\rightarrow (S^0,\nabla_x S^0) \quad  \text{in } L^p_{\text{loc}}(\mathbb{R}^d \times (0,\infty)),\quad 1 \leq p \leq \infty,
 \end{aligned}
\end{equation*}
where $F$ is the equilibrium distribution defined in \eqref{def_F} and $(\rho_1^0,\rho_2^0,S^0)$ the solution of \eqref{limit_equation}.
\end{theorem}

\begin{remark}
Note that $D_i$ is a diagonal and positive definite matrix,
thanks to the symmetry assumption on $V$, and $\chi_i[S]$ is bounded. This ensures that the macroscopic equation \eqref{limit_equation} subject to the previous initial condition admits 
a unique and global-in-time solution. We refer the reader to the appendix for details.
\end{remark}

\subsection{Formal derivation of drift-diffusion limits}\label{Formal_derivation}
For the sake of clarity, we first derive formally the limit equation \eqref{limit_equation}. We consider Hilbert expansions of $f_1^{\varepsilon},f_2^{\varepsilon}$~:
\begin{equation} 
f_i^{\varepsilon}=f_i^{0}+\varepsilon f_i^{1}+\eps^2 f_i^2+o(\varepsilon^2),\quad \mbox{for } i=1,2. \label{exp_f1}
\end{equation}
Assume $S^\eps=S^0$ is independent of $\eps$ and given by \eqref{eq_S_eps} 
with the right-hand side $\rho_1^0+\rho_2^0$ and consider in this part that
$\theta_i$ is smooth for $i=1,2$.

Injecting \eqref{exp_f1} into the equation for $f_i^{\varepsilon}$
\eqref{kinetics_eps}--\eqref{dev_T}--\eqref{def_T_ik} and identifying 
the terms in $O(1)$ and $O(\varepsilon)$ leads to,
\begin{equation*}
 \begin{aligned}
 & f_i^0=\frac{1}{\abs{V}}\int_V f_i^0 dv= \rho_i^0 F,\\
 & f_i^1=\rho_i^1-\frac{v\cdot \nabla_x f_i^0}{\abs{V}\psi_i}+\frac{1}{\abs{V}}\left(\int_V \theta_i(v'\cdot \nabla_x S^0)f_i^0(v')dv'-\abs{V}\theta_i(v\cdot \nabla_x S^0)f_i^0(v)\right),
 \end{aligned}
\end{equation*}
for $i=1,2$.
Replacing $f_i^0$ in the expression of $f_i^1$ yields
\begin{equation}\label{f_i1}
 f_i^1=\rho_i^1-\frac{v\cdot\nabla_x\rho_i^0}{\abs{V}^2\psi_i} +\frac{\rho_i^0}{\abs{V}}\left(\int_V\!\!\theta_i(v'\cdot \nabla_x S^0)\frac{dv'}{\abs{V}}-\theta_i(v\cdot \nabla_x S^0)\right),\quad \text{for } i=1,2.
\end{equation}
Then for the $O(\eps^2)$ term, we have that 
\begin{equation}\label{eq:f2}
-\psi_i (|V|f_i^2 - \rho_i^2) = \pa_t \rho_i^0 + v\cdot\nabla_x f_i^1
+ \psi_i U_i, \quad \mbox{ for }\ i=1,2,
\end{equation}
where
$$
\begin{array}{ll}
U_i(v) = &
\displaystyle |V| \theta_i(v\cdot\nabla_xS^0) f_i^1 - \int_V \theta_i(v'\cdot\nabla_xS^0) f_i^1(v')\,dv' \\[2mm]
& \displaystyle + \pa_tS^0\Big(|V| \theta'_i(v\cdot\nabla_xS^0) f_i^0 - \int_V \theta'_i(v'\cdot\nabla_xS^0) f_i^0(v')\,dv'\Big).
\end{array}
$$
We notice that $\int_V U_i \,dv= 0$.
Equation \eqref{eq:f2} admits a solution provided the integral over $V$
of the right-hand side vanishes. This implies the conservation law~:
 \begin{equation*}
  \partial_t \rho_i^0+ \nabla \cdot J_i^1=0,
 \end{equation*}
where $J_i^1:=\displaystyle \int_V vf_i^1 dv$, for $i=1,2$.
Using \eqref{f_i1} and formulas of $D_i$ and $\chi_i[S]$ \eqref{coef_limit_eq}, it follows that
\begin{equation*}
 J_i^1=-D_i \nabla_x \rho_i^0+\chi_i[S] \rho_i^0, \qquad \mbox{for } i=1,2.
\end{equation*}
This gives the equation for $\rho_i^0$ for $i=1,2$.

\section{Global existence of solutions of the kinetic model}\label{global_existence}
The purpose of this section is to prove the global existence for System \eqref{kinetics_eps}--\eqref{eq_S_eps}. The Green representation formula 
allows us to decouple \eqref{eq_S_eps} and \eqref{kinetics_eps}. This gives a system which depends only on $f_i^{\varepsilon}$. The fixed-point argument gives the uniqueness and local existence in time of solutions.
Thanks to a-priori estimates on $f_i^{\varepsilon}$, 
we recover global-in-time existence.
Without loss of generality, and for the sake of simplicity of the notation, we fix $\varepsilon=1$ and denote $\phi_i^{max}:=\max{\phi_i^{\varepsilon}}$.

\subsection{A-priori estimates}

We recall that using Bessel potential (see \cite{Evans}), 
the solution $S$ of elliptic/parabolic equation \eqref{eq_S2} is given by 
\begin{equation}
\begin{aligned}
S(x,t)&=\int_{\mathbb{R}^d}\!\!\! G(y)\left(\rho_1+\rho_2\right)(x-y,t) \,dy,\quad \delta=0,\\
S(x,t)&=\int_{0}^t\!\!\!\int_{\mathbb{R}^d}\!\!\! K(y,s)\left(\rho_1+\rho_2\right)(x-y,t-s) \,dy\,ds,\quad \delta=1,
\end{aligned}
\end{equation}
with
\begin{equation}\label{def_Kernel}
\begin{aligned}
 G(x)&:=\frac{1}{2}e^{-\abs{x}},\quad d=1,\\
 G(x)&:=\frac{1}{4\pi} \int_0^{\infty}\!\!\exp{\big(-\pi\frac{|x|^2}{4s}-\frac{s}{4\pi} \big)}\,s^{\tfrac{2-d}{2}}\frac{ds}{s},\quad d \geq 2,\\
 K(x,t)&:=\frac{1}{{(4\pi t)}^{d/2}}\exp{\big(-\frac{|x|^2}{4t}-t \big)},\quad d\geq 1.
\end{aligned}
 \end{equation}
We review some classical results on the integrability of kernels $G,K$ and their gradients.
\begin{lemma}[Estimates on $G$ and $K$]\label{properties_K}
Let $t>0 $. If $d=1$, then there exists a constant $C_1$ such that
\begin{align}
\left\|G\right\|_{L^1(\mathbb{R})}=1 &,\quad \int_0^t \left\| K(\cdot,s)\right\|_{L^1(\mathbb{R})} ds \leq 1,\\
\left\|\nabla_x G\right\|_{L^1(\mathbb{R})}=1 &,\quad \int_0^t \left\|\nabla_x K(\cdot,s)\right\|_{L^1(\mathbb{R})} ds \leq C_1\,t^{\frac{1}{2}}.\label{nabla_GK_1}
\end{align}
For $d\geq 2$, there exists constants $C_p,C_p^{'}$ such that
 \begin{align}
  \left\|G\right\|_{L^p(\mathbb{R}^d)} \leq &C_p ,\quad \int_0^t\left\| K(\cdot,s)\right\|_{L^p(\mathbb{R}^d)} ds \leq C_p \,t^{\frac{d(1-p)}{2p}+1},\quad 1\leq p<\frac{d}{d-2},\label{GK}\\
  \left\|\nabla_x G\right\|_{L^p(\mathbb{R}^d)} \leq &C_p^{'},\quad \int_0^t\left\|\nabla_x K(\cdot,s)\right\|_{L^p(\mathbb{R}^d)} ds \leq C_p^{'}\,t^{\frac{d(1-p)+p}{2p}},\quad 1\leq p<\frac{d}{d-1}.\label{nablaGK}
 \end{align}
\end{lemma}
\begin{proof}
For $d=1$, simple computations give the result for $G$ and $\nabla_x G$. For $\left\|K(\cdot,s)\right\|_{L^1(\mathbb{R})}$, the transformation  $y=\frac{x}{\sqrt{2s}}$ leads to
\begin{equation*}
 \left\|K(\cdot,t)\right\|_{L^1(\mathbb{R})}=e^{-t}.
\end{equation*}
It follows that 
\begin{equation*}
 \int_0^t\left\| K(\cdot,s)\right\|_{L^1(\mathbb{R})} ds \leq 1.
\end{equation*}
By similar computations, we show that
\begin{equation*}
 \left\|\nabla_x K(\cdot,s)\right\|_1=\frac{1}{{4\pi}^{d/2}}
\left\|y\mapsto ye^{-\frac{|y|^2}{2}}\right\|_{L^1(\mathbb{R})} s^{-\frac{1}{2}} e^{-s}.
 \end{equation*}
By integrating with respect to $s$ in $(0,t)$, we obtain the result.

We now suppose that $d\geq 2$ and compute $L^p$-norms of $G$ and $K$.
\begin{equation*}
  \left\|G\right\|_{L^p(\mathbb{R}^d)}=\frac{1}{4\pi} \int_0^{\infty} 
\left\|x\mapsto e^{-\pi\frac{|x|^2}{4s}}\right\|_{L^p(\mathbb{R}^d)}
\,e^{-\frac{s}{4\pi}}\, s^{\tfrac{2-d}{2}}\frac{ds}{s}.
\end{equation*}
Performing the change of variable $y=\sqrt{\frac{\pi}{2s}}x$ and simplifying yields
 \begin{equation*}
    \left\|G\right\|_{L^p(\mathbb{R}^d)}=\frac{2^{d(\frac{2}{p}-1)}}{p^{d/2p}\pi^{\frac{d}{2}(1-\frac{1}{p})}}\int_0^{\infty}\!\!\!e^{-u}\,u^{\frac{d}{2}(\frac{1}{p}-1)}\,du.
 \end{equation*}
After straightforward computations, we deduce that for all $1\leq p <\frac{d}{d-2}$
\begin{equation*}
\left\|G\right\|_{L^p(\mathbb{R}^d)} \leq  \frac{2^{d(\frac{2}{p}-1)}}{p^{d/2p}\pi^{\frac{d}{2}(1-\frac{1}{p})}} \Gamma \left(1+\frac{d-dp}{2p}\right).
\end{equation*}
A similar transformation applied to $\left\|K(\cdot,s)\right\|_{L^p(\mathbb{R}^d)}$ gives 
\begin{equation*}
 \left\|K(\cdot,s)\right\|_{L^p(\mathbb{R}^d)}=\frac{2^{d(\frac{1}{p}-1)}}{\pi^{\frac{d}{2}(1-\frac{1}{p})}} s^{\frac{d}{2}(\frac{1}{p}-1)}e^{-s}.
\end{equation*}
We conclude that for all $1\leq p <\frac{d}{d-2}$
 \begin{equation*}
  K(\cdot,s) \in L^p(\mathbb{R}^d)\quad \text{and }\int_0^t\!\!\!\left\| K(\cdot,s)\right\|_{L^p(\mathbb{R}^d)} ds \leq C \,t^{\frac{d(1-p)}{2p}+1}.
 \end{equation*}

For the estimates on the gradients, we have
 \begin{equation*}
 \left\| \nabla_x G\right\|_{L^p(\mathbb{R}^d)}=\frac{1}{4\pi} \int_0^{\infty} \frac{\pi}{2s} \left\|xe^{-\pi\frac{|x|^2}{4s}}\right\|_{L^p(\mathbb{R}^d)}\,e^{-\frac{s}{4\pi}}\,s^{\tfrac{2-d}{2}}\frac{ds}{s}.
  \end{equation*}
We apply the two successive transformations $y=\sqrt{\frac{\pi}{2s}}x,t=\frac{s}{4\pi}$ and simplify. We see  that for all $1\leq p <\frac{d}{d-1}$
  \begin{equation*}
  \left\|\nabla_x G \right\|_{L^p(\mathbb{R}^d)}=\frac{2^{3d/2p-(2d+3)/2}}{\pi^{d/2}}\Gamma\left(1/2-\frac{d}{2}(1-1/p)\right) \left\|ye^{-\frac{|y|^2}{2}}\right\|_{L^p(\mathbb{R}^d)}.
 \end{equation*}
Similar computations applied to  $\nabla_x K$ give
\begin{equation*}
 \left\|\nabla_x K(\cdot,s)\right\|_{L^p(\mathbb{R}^d)}=\frac{2^{\frac{d}{2}(1/p-1)-d}}{\pi^{d/2}}\left\|ye^{-\frac{|y|^2}{2}}\right\|_{L^p(\mathbb{R}^d)} s^{\frac{d}{2p}-\frac{d+1}{2}} e^{-s}.
 \end{equation*}
Integrating with respect to $s$ leads to the result.
\end{proof}

\begin{lemma}\label{diff_phi}
Fix $\tau>0$ and $v$ in $V$. Let $\phi$ be a Lipschitz continuous function and $f,\tilde{f}$ be in $L^{\infty}([0,\tau),L^{\infty}(\Omega)\cap L^1(\Omega))$ such that $f,\tilde{f}$ coincide at the time $t=0$ and satisfy in a weak sense
\begin{equation}\label{law_conservation}
 \partial_t \rho(f)+\nabla \cdot J(f)=0,
\end{equation}
where $J$ is a linear and bounded operator on $L^{\infty}(\Omega)$.
Let $S$ and $\tilde{S}$ denote 
  \begin{align*}
 S(x,\tau)&:=G*\rho(f),\quad \tilde{S}(x,\tau):=G*\rho(\tilde{f}),\quad \mbox{ for }\delta=0,\\
 S(x,\tau)&:=\int_0^{\tau}\!\!\!K(\cdot,s)*\rho(f)(\cdot,\tau-s)ds,\quad \tilde{S}(x,\tau):=\int_0^{\tau}\!\!\!K(\cdot,s)*\rho(\tilde{f})(\cdot,\tau-s)ds,\quad \mbox{ for }\delta=1.
 \end{align*}
Then, there exists a positive constant $C$ such that
\begin{equation*}
\|\phi(\partial_t S+v\cdot \nabla_x S)-\phi(\partial_t \tilde{S}+v\cdot \nabla_x \tilde{S})\|_{L^{\infty}(\Omega)}\leq C \left\|\phi\right\|_{C^{0,1}(\mathbb{R})}\left\|\nabla_x G\right\|_{L^1(\mathbb{R}^d)} \|(f-\tilde{f})(\cdot,\tau)\|_{L^{\infty}(\Omega)}
\end{equation*}
for $\delta=0$, and 

$$
\begin{array}{l}
\|\phi(\partial_t S+v\cdot \nabla_x S)-\phi(\partial_t \tilde{S}+v\cdot \nabla_x \tilde{S})\|_{L^{\infty}(\Omega)}  \\[2mm]
\qquad\qquad \leq 
\displaystyle C \|\phi\|_{C^{0,1}(\mathbb{R})}\int_0^{\tau}\!\!\!\|\nabla_x K(\cdot,s)\|_{L^1(\mathbb{R}^d)}\|(f-\tilde{f})(\cdot,\tau-s)\|_{L^{\infty}(\Omega)}ds,
\qquad \mbox{ for }\delta =1.
\end{array}
$$

\end{lemma}
\begin{proof}
For $\delta=0$, recalling the expression of $S=G*\rho(f)$, differentiating with respect to $t$ and using the conservation equation \eqref{law_conservation}, we get 
from Green's formula
\begin{equation*}
 \partial_t S=\int_{\mathbb{R}^d}\!\!\!\nabla_x G(y)\cdot J(f) (x-y,\tau)\,dy.
\end{equation*}
We proceed in the same way for $\nabla_x S$. Putting together $\partial_t S$ and $\nabla_x S$ terms, one obtains 
 \begin{equation*}
 \partial_t S+v\cdot \nabla_x S =\int_{\mathbb{R}^d} \nabla_x G(y)\cdot (J-v\rho)(f)(x-y,\tau)\,dy.
\end{equation*}
This formula combined with the Lipschitz continuity of $\phi$ implies
 \begin{equation*}
 \begin{aligned}
    \abs{\phi(\partial_t S+v\cdot \nabla_x S)-\phi(\partial_t \tilde{S}+v\cdot \nabla_x \tilde{S})}&\leq \|\phi\|_{C^{0,1}(\mathbb{R})}\abs{\partial_t(S-\tilde{S})+v\cdot \nabla_x (S-\tilde{S})}\\
    \phantom{ \abs{\phi(\partial_t S+v\cdot \nabla_x S)-\phi(\partial_t \tilde{S}+v\cdot \nabla_x \tilde{S})}}&\leq \|\phi\|_{C^{0,1}(\mathbb{R})}\abs{\nabla_x G*(J-v\rho)(f-\tilde{f})}.
 \end{aligned}
 \end{equation*}
By applying Young's inequality and using either \eqref{nabla_GK_1} or \eqref{nablaGK}, it follows that
\begin{equation*}
 \abs{\phi(\partial_t S+v\cdot \nabla_x S)-\phi(\partial_t \tilde{S}+v\cdot \nabla_x \tilde{S})} \leq \|\phi\|_{C^{0,1}(\mathbb{R})}\left\|\nabla_x G\right\|_{L^1(\mathbb{R}^d)}\|(J-v\rho)(f-\tilde{f})\|_{L^\infty(\Omega)}.
\end{equation*}
From the assumption, $J$ is bounded on $L^{\infty}(\Omega)$. Since $V$ is a bounded domain, the linear operator $\rho$ is also bounded on $L^{\infty}(\Omega)$. Then, we conclude that
\begin{equation*}
 \abs{\phi(\partial_t S+v\cdot \nabla_x S)-\phi(\partial_t \tilde{S}+v\cdot \nabla_x \tilde{S})} \leq C \|\phi\|_{C^{0,1}(\mathbb{R})}\left\|\nabla_x G\right\|_{L^1(\mathbb{R}^d)}\|(f-\tilde{f})(\cdot,\tau)\|_{L^{\infty}(\Omega)}.
\end{equation*}
The case $\delta=1$ is treated similarly. The slight difference comes from the additional term appearing in the expression of $\partial_t S$~: 
\begin{equation*}
  \partial_t S=\int_{\mathbb{R}^d}\!\!\!K(y,\tau)\rho(f)(x-y,0)\,dy+\int_0^{\tau}\!\!\!\int_{\mathbb{R}^d}\!\!\!K(s,y)\partial_t \rho(f)(x-y,\tau-s)\,dy \quad \text{in a weak sense}.
\end{equation*}
Then $\partial_t S+v\cdot \nabla_x S$ becomes
\begin{equation*}
\partial_t S+v\cdot \nabla_x S =\int_{\mathbb{R}^d}\!\!\!K(y,\tau)\rho (f)(x-y,0)\,dy+\int_0^{\tau}\!\!\!\int_{\mathbb{R}^d}\!\!\!\nabla_x K(y,s)(J-v\rho)(f)(x-y,\tau-s)\,dy\,ds.
\end{equation*}
The substraction between $\partial_t S+v\cdot \nabla_x S$ and $\partial_t \tilde{S}+v\cdot \nabla_x \tilde{S}$ has the same form as in the elliptic setting since $f(\cdot,0)=\tilde{f}(\cdot,0)$.
Therefore,
\begin{equation*}
 \abs{\phi(\partial_t S+v\cdot \nabla_x S)-\phi(\partial_t \tilde{S}+v\cdot \nabla_x \tilde{S})} \leq C \|\phi\|_{C^{0,1}(\mathbb{R})}\int_0^{\tau}\!\!\!\|\nabla_x K(\cdot,s)\|_{L^1(\mathbb{R}^d)}\|(f-\tilde{f})(\cdot,\tau-s)\|_{L^{\infty}(\Omega)}ds.
\end{equation*}
\end{proof}

\begin{lemma}[A-priori bounds on $f_1,f_2$]\label{estimate}
Let $\tau>0$ and $(f_1,f_2)$ be a weak solution of \eqref{kinetics}
such that $f_1,f_2$ are  in $L^1((0,\tau),L^1_{+} \cap L^{\infty}(\Omega))$. We assume that tumbling rates $T_1[S],T_2[S]$ defined by \eqref{tumbling_kernel} are positive and bounded.

If the inital data $(f^{ini}_1,f^{ini}_2)$ belongs to $L^{\infty}(\Omega)\times L^{\infty}(\Omega)$, then there exists a constant $C>0$ such that for $i=1,2$ and $t\in (0,\tau)$, we have 
 \begin{equation*}
  \begin{aligned}
    & \left\|f_i(\cdot,t)\right\|_{L^1(\Omega)}=\left\|f^{ini}_i\right\|_{L^1(\Omega)},\\
    &  \left\|f_i(\cdot,t)\right\|_{L^{\infty}((0,\tau),L^{\infty}(\Omega))}\leq C \left\|f^{ini}_i\right\|_{L^{\infty}(\Omega)}e^{\abs{V} \phi_i^{max}\tau}.
  \end{aligned}
 \end{equation*} 
\end{lemma}
\begin{proof}
By integrating the equation for $f_i$ \eqref{kinetics}  with respect to $v$, we see that $\rho_i$ satisfies the conservation law~:
\begin{equation}\label{cell_conservation}
 \partial_t \rho_i +\nabla \cdot J_i=0,
\end{equation}
with $\displaystyle J_i:=\int_V v f_i$.
It follows that the $L^1$-norm of $\rho_i$ is conserved. We show the second inequality by using the Duhamel representation formula and the Gronwall lemma. 
Since each equation for $i=1$ and $2$ can be treated separately, the proof is 
identical to the single-species case. We refer the reader to \cite{NV}.
\end{proof}

\subsection{Proof of Theorem \ref{existence}}

We now prove the global-in-time existence.
It is standard that if they exist, the solutions $f_1^\eps$ and $f_2^\eps$ 
are nonnegative provided the initial data are nonnegative.
Since $\rho_i$ satisfies the conservation law \eqref{cell_conservation}, by the proof of Lemma~\ref{diff_phi}, $\partial_t S+v\cdot \nabla_x S$ is given for $\delta=0$ and $\delta=1$, respectively, by
\begin{equation*}
 \begin{aligned}
 \partial_t S+v\cdot \nabla_x S &=\int_{\mathbb{R}^d} \nabla_x G(y)\cdot (J-v\rho)(f_1+f_2)(x-y,t)\,dy,\\
 \partial_t S+v\cdot \nabla_x S & =\int_{\mathbb{R}^d}\!\!\!K(y,t)\rho (f_1+f_2)(x-y,0)\,dy\\
 & +\int_0^t\!\!\!\int_{\mathbb{R}^d}\!\!\!\nabla_x K(y,s)(J-v\rho)(f_1+f_2)(x-y,t-s)\,dy\,ds.
 \end{aligned}
\end{equation*}
Replacing $\partial_t S+ v\cdot \nabla_x S$ in \eqref{kinetics} yields in the case $\delta=0$
\begin{equation}\label{syst_ell}
 \left\{
\begin{aligned}
 \partial_t f_i+v\cdot \nabla_x f_i&=\int_{V}\!\!\!\phi_i\left(\nabla_x G*(J-v'\rho)(f_1+f_2)\right)f_i^{'}dv'\\
 &-\abs{V}\phi_i\left(\nabla_x G*(J-v\rho)(f_1+f_2)\right)f_i,\\
 f_i(x,v,t=0)&=f_i^{ini},\quad\text{for }i=1,2.
 \end{aligned}
 \right.
\end{equation}
 For $\delta=1$, we obtain
 \begin{equation}\label{syst_par}
  \left\{
  \begin{aligned}
    \partial_t f_i  +v\cdot \nabla_x f_i &=\int_{V}\phi_i\left(K*\rho (f_1^{ini}+f_2^{ini})+\!\int_0^t\!\!\!\nabla_x K(\cdot,s)*(J-v'\rho)(f_1+f_2)(\cdot,t-s)ds\right)f_i^{'}dv'\\
  &-\abs{V}\phi_i\left(K*\rho (f_1^{ini}+f_2^{ini})+\!\int_0^t\!\!\!\nabla_x K(\cdot,s)*(J-v\rho)(f_1+f_2)(\cdot,t-s)ds\right)f_i,\\
    f_i(x,v,t=0)&=f_i^{ini},\quad\text{for }i=1,2.
   \end{aligned}
   \right.
  \end{equation}
Here, we define the convolution between two vector-valued functions $M,H:\mathbb{R}\rightarrow \mathbb{R}^d$ as
\begin{equation*}
 M*H=\sum_{j=0}^d M_j*H_j,
\end{equation*}
where $M_j,H_j$ are components of $M$ and $H$.

We fix $\tau>0$ and introduce the Banach space $(X^{\tau},\|\cdot \|_{X^{\tau}})$ given by
\begin{equation*}
 \begin{aligned}
  X^{\tau}:&=L^{1}((0,\tau),L^{\infty}(\Omega))\times L^{1}((0,\tau),L^{\infty}(\Omega)),\\
  \|f\|_{X^{\tau}}:&=\int_0^{\tau} \left(\|f_1(\cdot,s)\|_{L^{\infty}(\Omega)}+\|f_2(\cdot,s)\|_{L^{\infty}(\Omega)}\right)ds
 \quad \text{for }f=(f_1,f_2) \in X^{\tau}.
 \end{aligned}
\end{equation*}
We build fixed-point operators $\mathcal{F}(f)=(\mathcal{F}_1(f), \mathcal{F}_2(f))$ 
for Systems \eqref{syst_ell}, \eqref{syst_par} on $X^{\tau}$.
Since these systems are different, we need to consider two cases and treat them separately.

\subsubsection{Proof of the elliptic case, $\delta=0$}
$(\mathcal{F}_1(f), \mathcal{F}_2(f))$ is the weak solution of the system:
\begin{equation}\label{def_F_e}
 \left\{
 \begin{aligned}
  \partial_t \mathcal{F}_i+v\cdot \nabla_x \mathcal{F}_i&=\int_{V}\!\!\!\phi_i\left(\nabla_x G*(J-v'\rho)(f_1+f_2)\right)\mathcal{F}_i^{'}\,dv'\\
  &-\abs{V} \phi_i\left(\nabla_x G*(J-v\rho)(f_1+f_2)\right)\mathcal{F}_i,\\
  \mathcal{F}_i(\cdot,t=0)&=f_i^{ini},\quad \text{for } i=1,2,
 \end{aligned}
 \right.
\end{equation}
with $\mathcal{F}_i:=\mathcal{F}_i(f)(x,v,t)\, \text{ and }\mathcal{F}_i^{'}:=\mathcal{F}_i(f)(x,v',t)$.

\medskip
For $f=(f_1,f_2)\,\text{and }g=(g_1,g_2)$ in $X^{\tau}$, we define the mapping $\mathcal{F}^{fg}:= \mathcal{F}(f)-\mathcal{F}(g)$ whose components $\mathcal{F}_i^{fg}$ are defined by
\begin{equation*}
 \mathcal{F}_i^{fg}:=\mathcal{F}_i(f)-\mathcal{F}_i(g),\quad \text{for } i=1,2.
\end{equation*}
Substracting equations for $\mathcal{F}_i(f)$ and $\mathcal{F}_i(g)$ and collecting terms leads to
\begin{equation}\label{eq_fp}
 \left\{
 \begin{aligned}
  &\partial_t \mathcal{F}_i^{fg}  +v\cdot \nabla_x \mathcal{F}_i^{fg}+ 
\abs{V} \phi_i(\nabla_x G*(J-v\rho)(f_1+f_2))\mathcal{F}_i^{fg}=\mathcal{G}_i^{fg},\\
  & \mathcal{F}_i^{fg}(\cdot,t=0)=0,\quad \text{for } i=1,2,
 \end{aligned}
 \right.
\end{equation}
where $\mathcal{G}_i^{fg}$ is defined by 
\begin{multline}\label{def_G}
 \mathcal{G}_i^{fg}(x,v,t):=\int_{V}\phi_i\left(\nabla_x G*(J-v'\rho)(f_1+f_2)\right)\mathcal{F}_i^{fg}(v')\,dv' \\
 -\mathcal{F}_i(g)\abs{V}\left(\phi_i \left(\nabla_x G*(J-v\rho)(f_1+f_2)\right) 
 -\phi_i\left(\nabla_x G*(J-v\rho)(g_1+g_2)\right) \right) \\
  +\int_{V}\mathcal{F}_i(g)(v')\left(\phi_i\left(\nabla_x G*(J-v'\rho)(f_1+f_2)\right) 
 - \phi_i\left(\nabla_x G*(J-v'\rho)(g_1+g_2)\right)\right)dv'. 
\end{multline}

Thanks to the Duhamel formula, $\mathcal{F}_i^{fg}$ writes
\begin{multline}\label{Duh}
 \mathcal{F}_i^{fg}(x,v,t)=\int_0^t\!\!\exp \left(-\int_s^t \phi_i(\nabla_x G*(J-v\rho)(f_1+f_2))(x-v(s-u),u)du\right)\\
 \times \mathcal{G}_i^{fg}(x-v(t-s),v,s)\,ds.
\end{multline}
We would like to bound $\displaystyle \int_0^{\tau}\|\mathcal{F}_1^{fg}(\cdot,t)\|_{L^{\infty}(\Omega)}dt$ and $\displaystyle \int_0^{\tau}\|\mathcal{F}_2^{fg}(\cdot,t)\|_{L^{\infty}(\Omega)}dt$ by the $X^{\tau}$-norm of $f-g$.
We deal with the first term.
By using the boundness of $\phi_1$, we get
\begin{align*}
 & \abs{\mathcal{G}_1^{fg}(x,v,s)}\leq \phi_1^{max}\int_{V}\abs{\mathcal{F}_1^{fg}(v')\,dv'}\\
 & \phantom{\abs{\mathcal{G}_1^{fg}(x,v,s)} \leq }+\abs{\mathcal{F}_1(g)}\abs{V}\abs{\phi_1(\nabla_x G*(J-v\rho)(f_1+f_2))-\phi_1(\nabla_x G*(J-v\rho)(g_1+g_2))}\\
 & \phantom{\abs{\mathcal{G}_1^{fg}(x,v,s)} \leq  }+\int_{V}\abs{\mathcal{F}_1(g)(v')}\abs{\phi_1(\nabla_x G*(J-v'\rho)(f_1+f_2))- \phi_1(\nabla_x G*(J-v'\rho)(g_1+g_2))}dv'. 
\end{align*}
Since $\phi_1$ is a Lipschitz continuous function, by applying Lemma~\ref{diff_phi} with $f=f_1+f_2$ and $g=g_1+g_2$, we find that
\begin{multline*}
\abs{\mathcal{G}_1^{fg}(x,v,s)}\leq \phi_1^{max} \int_{V}\abs{\mathcal{F}_1^{fg}(x,v',s)}dv'\\
+C \|\phi_1\|_{C^{0,1}(\mathbb{R})}\|\nabla_x G\|_{L^1(\mathbb{R}^d)}\left(\abs{\mathcal{F}_1(g)}+\int_V\abs{\mathcal{F}_1(g)}\right) 
\sum_{i=1,2}\left\|(f_i-g_i)(\cdot,s)\right\|_{L^{\infty}(\Omega)}.
\end{multline*}
Hence, 
\begin{multline*}
 \left\|\mathcal{G}_1^{fg}(\cdot,s)\right\|_{L^{\infty}(\Omega)} \leq \phi_1^{max} \abs{V} \|\mathcal{F}_1^{fg}(\cdot,s)\|_{L^{\infty}(\Omega)} \\
+ C \big(1+\abs{V}\big) \|\phi_1\|_{C^{0,1}(\mathbb{R})}\|\nabla_x G\|_{L^1(\mathbb{R}^d)}
\|\mathcal{F}_1(g)(\cdot,s)\|_{L^{\infty}(\Omega)}\sum_{i=1,2}\left\|(f_i-g_i)(\cdot,s)\right\|_{L^{\infty}(\Omega)} .
\end{multline*}
Since $\phi_1$ is nonnegative, from \eqref{Duh} we have 
\begin{equation*}
 \abs{\mathcal{F}_1^{fg}(x,v,t)} \leq \int_0^t \abs{\mathcal{G}_1^{fg}(x-v(t-s),v,s)}ds.
\end{equation*}
Taking the $L^{\infty}$-norm on $\Omega$ of both sides gives
\begin{equation*}
 \left\|\mathcal{F}_1^{fg}(\cdot,t)\right\|_{L^{\infty}(\Omega)}\leq \int_0^t\left\|\mathcal{G}_1^{fg}(\cdot,s)\right\|_{L^{\infty}(\Omega)}.
\end{equation*}
Recalling the estimate on $\left\|\mathcal{G}_1^{fg}(\cdot,s)\right\|_{L^{\infty}(\Omega)}$ above, it follows that
\begin{multline*}
 \left\|\mathcal{F}_1^{fg}(\cdot,t)\right\|_{L^{\infty}(\Omega)}  \leq \phi_1^{max} \abs{V}\int_0^t \|\mathcal{F}_1^{fg}(\cdot,s)\|_{L^{\infty}(\Omega)}ds\\
 +C_1\int_0^t \left\|\mathcal{F}_1(g)(\cdot,s)\right\|_{L^{\infty}(\Omega)}
 \sum_{i=1,2}\left\|(f_i-g_i)(\cdot,s)\right\|_{L^{\infty}(\Omega)}ds.
\end{multline*}
The bound on $\left\|\mathcal{F}_1(g)(\cdot,s)\right\|_{L^{\infty}(\Omega)}$ is similar to the one given by Lemma~\ref{estimate}. By using this estimate, we get
\begin{multline*}
 \left\|\mathcal{F}_1^{fg}(\cdot,t)\right\|_{L^{\infty}(\Omega)}\leq \phi_1^{max} \abs{V} \int_0^t \left\|\mathcal{F}_1^{fg}(\cdot,s)\right\|_{L^{\infty}(\Omega)}ds\\
 +C_1^{'} e^{\phi_1^{max} \abs{V}t} \int_0^t \sum_{i=1,2}\left\|(f_i-g_i)(\cdot,s)\right\|_{L^{\infty}(\Omega)}ds.
\end{multline*}
The Gronwall lemma asserts that
\begin{align*}
 \int_0^{\tau} \left\|\mathcal{F}_1^{fg}(\cdot,t)\right\|_{L^{\infty}(\Omega)}dt
 &\leq C_1^{'}\big(e^{\phi_1^{max} \abs{V}\tau}-1\big)\int_0^{\tau}\sum_{i=1,2}\left\|(f_i-g_i)(\cdot,t)\right\|_{L^{\infty}(\Omega)}dt. 
\end{align*}
We obtain a similar estimate on $\mathcal{F}_2^{fg}$ by replacing $\phi_1^{max}$ by $\phi_2^{max}$.
Summing the estimates on $\mathcal{F}_i^{fg}$ for $i=1,2$, we deduce that
\begin{equation*}
 \int_0^{\tau} \sum_{i=1,2}\|\mathcal{F}_i^{fg}(\cdot,t)\|_{L^{\infty}(\Omega)} dt \leq C \big(e^{\phi^{max}\tau}-1\big)\int_0^{\tau} \sum_{i=1,2}\left\|(f_i-g_i)(\cdot,t)\right\|_{L^{\infty}(\Omega)}dt,
\end{equation*}
with $\phi^{max}=\max{(\phi_1^{max},\phi_2^{max})}$.
Recalling the definition of the $X^{\tau}$-norm, we conclude that
\begin{equation*}
 \left\| \mathcal{F}(f)-\mathcal{F}(g)\right\|_{X^{\tau}} \leq C \big(e^{\phi^{max}\tau}-1\big) \|f-g\|_{X^{\tau}}.
\end{equation*}
Therefore, there exists a sufficiently small time $\tau_0>0$  such that $\mathcal{F}$ is a contraction mapping  on $X^{\tau_0}=L^{1}((0,\tau_0),L^{\infty}(\Omega))\times L^{1}((0,\tau_0),L^{\infty}(\Omega))$.
Thus, Banach fixed-point Theorem gives the existence and uniqueness of weak solution $f=(f_1,f_2)$ of \eqref{kinetics_eps}--\eqref{eq_S_eps} in $X^{\tau_0}$.

\medskip
From the a-priori estimates of Lemma \ref{estimate}, we see that $f(\cdot,\tau_0)$ is bounded in $L^1(\Omega) \cap L^{\infty}(\Omega)$. We now consider our problem with the initial condition $f(\cdot,\tau_0)$, apply the same strategy and obtain the existence up to the time $2\tau_0$.
By iterating this procedure, we extend the solution to $L^{1}((0,\tau),L^{\infty}(\Omega))\times L^{1}((0,\tau),L^{\infty}(\Omega))$.

\subsubsection{Proof of the parabolic case, $\delta=1$}
In this case, $(\mathcal{F}_1(f),\mathcal{F}_2(f))$ verifies
\begin{equation}\label{def_F_p}
 \left\{
 \begin{aligned}
 & \partial_t \mathcal{F}_i+v\cdot \nabla_x \mathcal{F}_i=\int_{V}\phi_i\left(\int_0^t\!\!\nabla_x K(\cdot,s)*(J-v'\rho)(f_1+f_2)(\cdot,t-s)\,ds\right)\mathcal{F}_i^{'}dv,'\\
 & \phantom{\partial_t \mathcal{F}_i+v\cdot \nabla_x \mathcal{F}_i=}
  -\abs{V}\phi_i\left(\int_0^t \nabla_x K(\cdot,s)*(J-v\rho)(f_1+f_2)(\cdot,t-s)\,ds\right)\mathcal{F}_i,\\
 & \mathcal{F}_i(\cdot,t=0)=f_i^{ini},\quad\text{for } i=1,2.
 \end{aligned}
 \right.
\end{equation}
Therefore, $(\mathcal{F}_1^{fg},\mathcal{F}_2^{fg})$ satisfies
\begin{equation*}\label{eq_fe}
 \left\{
 \begin{aligned}
  &\partial_t \mathcal{F}_i^{fg}  +v\cdot \nabla_x \mathcal{F}_i^{fg}+\abs{V}\phi_i
  \left(\int_0^t \nabla_x K(\cdot,s)*(J-v\rho)(f_1+f_2)(\cdot,t-s)\,ds\right) \mathcal{F}_i^{fg}=\mathcal{G}_i^{fg},\\
  & \mathcal{F}_i^{fg}(\cdot,t=0)=0,\quad \text{for } i=1,2,
 \end{aligned}
 \right.  
\end{equation*}
where $\mathcal{G}_i^{fg}$ is defined by 
\begin{multline*}
 \mathcal{G}_i^{fg}(x,v,t):=\int_{V}\phi_i(\int_0^t\!\!\nabla_x K(\cdot,s)*(J-v'\rho)(f_1+f_2)(\cdot,t-s)ds)\mathcal{F}_i^{fg}(v')dv'\\
-\mathcal{F}_i(g)(\phi_i (\int_0^t\!\!\nabla_x K(\cdot,s)*(J-v\rho)(f_1+f_2)(\cdot,t-s)ds))-\phi_i(\int_0^t\!\!\nabla_x K(\cdot,s)*(J-v\rho)(g_1+g_2)(\cdot,t-s)ds))\\
  +\mathcal{F}_i(g)(v')(\phi_i(\int_0^t\!\!\nabla_x K(\cdot,s)*(J-v'\rho)(f_1+f_2)(\cdot,t-s)ds)-\phi_i(\int_0^t\!\!\nabla_x K(\cdot,s)*(J-v'\rho)(g_1+g_2)(\cdot,t-s)ds))dv'.
\end{multline*}
By the Duhamel formula, we obtain the expression of $\mathcal{F}_i^{fg}(x,v,t)$
\begin{multline}\label{Duhamel}
 \mathcal{F}_i^{fg}(x,v,t)=\!\!\int_0^t \exp \left(-\int_s^t \phi_i\left(\int_0^u\!\!\nabla_x K(\cdot,r)*(J-v\rho)(f_1+f_2)(x-v(s-u),u-r)dr\right)du \right)\\
\times  \mathcal{G}_i^{fg}(s,x-v(t-s),v) \,ds.
\end{multline}
Using the boundness of $\phi_1$ and applying Lemma~\ref{diff_phi}, we get  
\begin{multline*}
 \abs{\mathcal{G}_1^{fg}(x,v,s)}  \leq \phi_1^{max}\int_{V}\abs{\mathcal{F}_1^{fg}(x,v',s)}dv'+C\big(\abs{\mathcal{F}_1(g)}+\int_V\abs{\mathcal{F}_1(g)}\big)\|\phi_1\|_{C^{0,1}(\mathbb{R})}\\ 
 \times \int_{0}^s\!\!\|\nabla_x K(\cdot,r)\|_{L^1(\mathbb{R}^d)} \sum_{i=1,2}\left\|(f_i-g_i)(\cdot,s-r)\right\|_{L^{\infty}(\Omega)}dr.
\end{multline*}
Taking the $L^{\infty}$-norm on $\Omega$ of both sides gives
\begin{multline}
\label{calGbound}
 \|\mathcal{G}_1^{fg}(\cdot,s)\|_{L^{\infty}(\Omega)}\leq  \phi_1^{max}\abs{V}\|\mathcal{F}_1^{fg}(\cdot,s)\|_{L^{\infty}(\Omega)}+C(1+\abs{V})\|\phi_1\|_{C^{0,1}(\mathbb{R})} \|\mathcal{F}_1(g)(\cdot,s)\|_{L^{\infty}(\Omega)}\\
 \times \int_{0}^s\!\!\|\nabla_x K(\cdot,r)\|_{L^1(\mathbb{R}^d)} \sum_{i=1,2}\left\|(f_i-g_i)(\cdot,s-r)\right\|_{L^{\infty}(\Omega)}dr.
\end{multline}
From \eqref{Duhamel}, it is clear that
 \begin{equation*}
 \abs{\mathcal{F}_1^{fg}(x,v,t)} \leq \int_0^t \abs{\mathcal{G}_1^{fg}(x-v(t-s),v,s)}ds.
\end{equation*}
It follows that
 \begin{equation*}
 \|\mathcal{F}_1^{fg}(\cdot,t)\|_{L^{\infty}(\Omega)} \leq \int_0^t\|{\mathcal{G}_1^{fg}(\cdot,s)}\|_{L^{\infty}(\Omega)}ds.
\end{equation*}
Using the bound on $\|{\mathcal{G}_1^{fg}(\cdot,s)}\|_{L^{\infty}(\Omega)}$ 
\eqref{calGbound} yields
\begin{multline*}
  \left\|\mathcal{F}_1^{fg}(\cdot,t)\right\|_{L^{\infty}(\Omega)}  \leq \phi_1^{max}\abs{V}\int_0^t \left\|\mathcal{F}_1^{fg}(\cdot,s)\right\|_{L^{\infty}(\Omega)}ds+ C_1\int_0^t \left\|\mathcal{F}_1(g)(\cdot,s)\right\|_{L^{\infty}(\Omega)}\\
 \times \left(\int_{0}^s\|\nabla_x K(\cdot,r)\|_{L^1(\mathbb{R}^d)} \sum_{i=1,2}\left\|(f_i-g_i)(\cdot,s-r)\right\|_{L^{\infty}(\Omega)}dr\right) ds. 
\end{multline*}
By using the estimate on $\left\|\mathcal{F}_1(g)(\cdot,s)\right\|_{L^{\infty}(\Omega)}$ which is the same as in Lemma~\ref{estimate}, we have
\begin{multline*}
\left\|\mathcal{F}_1^{fg}(\cdot,t)\right\|_{L^{\infty}(\Omega)}\leq \phi_1^{max}\abs{V} \int_0^t \left\|\mathcal{F}_1^{fg}(\cdot,s)\right\|_{L^{\infty}(\Omega)}ds \\
+ C_1^{'} e^{\phi_1^{max}\abs{V} t}\int_0^t \int_{0}^s\|\nabla_x K(\cdot,r)\|_{L^1(\mathbb{R}^d)} \sum_{i=1,2}\left\|(f_i-g_i)(s-r,\cdot,\cdot)\right\|_{L^{\infty}(\Omega)}dr\,ds. 
\end{multline*}
We bound the integral over $(0,s)$ by the integral over $(0,t)$. By applying a change of variable and the Fubini Theorem, it follows that
\begin{multline*}
\left\|\mathcal{F}_1^{fg}(\cdot,t)\right\|_{L^{\infty}(\Omega)}\leq \phi_1^{max}\abs{V}\int_0^t \left\|\mathcal{F}_1^{fg}(\cdot,s)\right\|_{L^{\infty}(\Omega)}ds \\
+ C_1^{'} e^{\phi_1^{max}\abs{V}t} \left(\int_0^t \|\nabla_x K(\cdot,r)\|_{L^1(\mathbb{R}^d)} dr \right)\left(\int_{0}^t \sum_{i=1,2}\left\|(f_i-g_i)(\cdot,s)\right\|_{L^{\infty}(\Omega)}ds\right). 
\end{multline*}
By Lemma~\ref{properties_K}, $\|\nabla_x K(\cdot,r)\|_{L^1(\mathbb{R}^d)}$ is integrable and we have
\begin{multline*}
 \left\|\mathcal{F}_1^{fg}(\cdot,t)\right\|_{L^{\infty}(\Omega)}\leq \phi_1^{max}\abs{V} \int_0^t \left\|\mathcal{F}_1^{fg}(\cdot,s)\right\|_{L^{\infty}(\Omega)}ds\\
+ C_1^{'}e^{\phi_1^{max} \abs{V}t}t^{1/2}\int_{0}^t \sum_{i=1,2}\left\|(f_i-g_i)(\cdot,s)\right\|_{L^{\infty}(\Omega)}ds.
\end{multline*}
Applying the Gronwall Lemma gives a bound on $\left\|\mathcal{F}_1^{fg}(\cdot,t)\right\|_{L^{\infty}(\Omega)}$ and the rest of the proof is analogous to the elliptic case.

\section{Rigorous proof of Drift-diffusion limit}\label{drift_diffusion}

In this section, we investigate the diffusive limit of the kinetic model, i.e we prove Theorem~\ref{limit}. We start by giving estimates on $S$ for given functions $\rho_1,\rho_2$ and useful inequalities. Afterwards, we state a proposition which gives uniform estimates on $f_1^{\varepsilon},f_2^{\varepsilon},S^{\varepsilon}$.
Finally, we conclude by using Aubin-Lions-Simon compactness lemma~\cite{JS}.
\subsection{A-priori estimates}
\begin{lemma}[Estimates on $S^{\varepsilon}$]\label{estimate_S}
Fix $\tau>0$. Let $p,q,\alpha$ be such that $1<p<\infty, d<q<\infty, 0<\alpha<1-\frac{d}{q}$. Assume that $\rho_1^{\varepsilon},\rho_2^{\varepsilon}\in L^{\infty}([0,\tau],L^1(\mathbb{R}^d)\cap L^{q}(\mathbb{R}^d))$.
Then, we have
\begin{itemize}
 \item 
 For $\delta=0$, $S^{\varepsilon}(\cdot,t)$ is bounded in $C^{1,\alpha}(\mathbb{R}^d)$ for all $t  \in [0,\tau]$ and
there exists a constant $c$ independent of $\varepsilon$ and $t$ such that
\begin{equation*}
 \|S^{\varepsilon}(\cdot,t)\|_{C^{1,\alpha}(\mathbb{R}^d)}\leq c\left(\|(\rho_1^{\varepsilon}+\rho_2^{\varepsilon})(\cdot,t)\|_{L^1(\mathbb{R}^d)}+ \|(\rho_1^{\varepsilon}+\rho_2^{\varepsilon})(\cdot,t)\|_{L^q(\mathbb{R}^d)}\right).
\end{equation*}
 \item 
For $\delta=1$, $S^{\varepsilon}$ is bounded in $L^{\infty}([0,\tau],C^{1,\alpha}(\mathbb{R}^d))\cap L^{\infty}([0,\tau],W^{1,p}(\mathbb{R}^d))$ and
there exists a constant $c_{\tau}$ independent of $\varepsilon$ such that
\begin{equation*}
 \|S^{\varepsilon}\|_{L^{\infty}([0,\tau],W^{1,p}(\mathbb{R}^d))}+\|S^{\varepsilon}\|_{L^{\infty}([0,\tau],C^{1,\alpha}(\mathbb{R}^d))}\leq c_{\tau}
 \big(\|\rho_1^{\varepsilon}+\rho_2^{\varepsilon}\|_{L^{\infty}([0,\tau],L^1(\mathbb{R}^d))}+ \|\rho_1^{\varepsilon}+\rho_2^{\varepsilon}\|_{L^{\infty}([0,\tau],L^q(\mathbb{R}^d))}\big).
\end{equation*}
\end{itemize}
\end{lemma}
\begin{proof}
The result for the case $\delta=0$ is classical and follows from elliptic regularity.
We refer the reader to \cite[chapter 9]{Brezis} for details.
We just give main arguments and derive bounds for $S^{\varepsilon}$. 
By the elliptic regularity, we know that $S^{\varepsilon}(\cdot,t)$ is 
bounded in $W^{2,q}(\mathbb{R}^d)$ and the following estimate holds.
\begin{equation*}
 \|S^{\varepsilon}(\cdot,t)\|_{W^{2,q}(\mathbb{R}^d)}\leq C\,\|(\rho_1^{\varepsilon}+\rho_2^{\varepsilon})(\cdot,t) \|_{L^q(\mathbb{R}^d)}.
\end{equation*}
By the Morrey theorem, $W^{2,q}$ is continuously embedded into $C^{1,\gamma}$ with $\gamma=1-\frac{d}{q}$.
\begin{equation*}
 W^{2,q}(\mathbb{R}^d)\hookrightarrow C^{1,\gamma}(\mathbb{R}^d) .
\end{equation*}
The interpolation between $L^1$ and $L^q$ implies that for all $d \leq p \leq q$,
$(\rho_1^{\varepsilon}+\rho_2^{\varepsilon})(\cdot,t)$ are bounded in $L^p$ and by elliptic regularity $S^{\varepsilon}(\cdot,t)$ is bounded in $W^{2,p}$. 
Thus, the $W^{2,p}$-norm of $S^{\varepsilon}$ is bounded by $L^p$-norms of $\rho_1^{\varepsilon}$ and $\rho_2^{\varepsilon}$.
\begin{equation*}
  \|S^{\varepsilon}(\cdot,t)\|_{W^{2,p}(\mathbb{R}^d)}\leq 
C \,\|(\rho_1^{\varepsilon}+\rho_2^{\varepsilon})(\cdot,t) \|_{L^p(\mathbb{R}^d)}.
\end{equation*}
By the Morrey Theorem, we conclude that for all $0<\alpha<1-\frac{d}{q}$, $S(\cdot,t)$ belongs to $C^{1,\alpha}$ and we have the estimate
\begin{equation*}
 \|S^{\varepsilon}(\cdot,t)\|_{C^{1,\alpha}(\mathbb{R}^d)} \leq 
C \, \|S^{\varepsilon}(\cdot,t)\|_{W^{2,p}(\mathbb{R}^d)},
\end{equation*}
where $p$ is given by $1-\frac{d}{p}=\alpha$.
\medskip\\
We now deal with the case $\delta=1$. 
We recall that $S$ is given by the convolution with the kernel $K$ defined in \eqref{def_Kernel}.
Applying the Young inequality and using Lemma~\ref{properties_K} yields for all $t$ in $[0,\tau]$
\begin{equation*}
 \begin{aligned}
  &\|S^{\varepsilon}(\cdot,t)\|_{L^{\infty}(\mathbb{R}^d)}\leq \left(\int_0^{\tau}\|K(\cdot,s)\|_{L^{q'}(\mathbb{R}^d)}\,ds\right) \|\rho_1^{\varepsilon}+\rho_2^{\varepsilon}\|_{L^{\infty}([0,\tau],L^q(\mathbb{R}^d))},\\
  &\|\nabla S^{\varepsilon}(\cdot,t)\|_{L^{\infty}(\mathbb{R}^d)}\leq \left(\int_0^{\tau}\|\nabla K(\cdot,s)\|_{L^{q'}(\mathbb{R}^d)}\,ds\right) \|\rho_1^{\varepsilon}+\rho_2^{\varepsilon}\|_{L^{\infty}([0,\tau],L^q(\mathbb{R}^d))}.
 \end{aligned}
\end{equation*}
Similarly, we have
\begin{equation*}
 \begin{aligned}
  &\|S^{\varepsilon}(\cdot,t)\|_{L^1(\mathbb{R}^d)}\leq \left(\int_0^{\tau}\|K(\cdot,s)\|_{L^1(\mathbb{R}^d)}\,ds\right) \|\rho_1^{\varepsilon}+\rho_2^{\varepsilon}\|_{L^{\infty}([0,\tau],L^1(\mathbb{R}^d))},\\  
  & \|\nabla S^{\varepsilon}(\cdot,t)\|_{L^1(\mathbb{R}^d)}\leq \left(\int_0^{\tau}\|\nabla K(\cdot,s)\|_{L^1(\mathbb{R}^d)}\,ds\right) \|\rho_1^{\varepsilon}+\rho_2^{\varepsilon}\|_{L^{\infty}([0,\tau],L^1(\mathbb{R}^d))}.
 \end{aligned}
\end{equation*}
By an interpolation argument, we deduce that $S^{\varepsilon}$ is bounded 
in $L^{\infty}([0,\tau],W^{1,p}(\mathbb{R}^d))$ for any $p$ 
between $1$ and $\infty$.
Thanks to the Morrey theorem, $S^{\varepsilon}$ belongs to 
$L^{\infty}([0,\tau],C^{0,\alpha}(\mathbb{R}^d))$ with 
$0<\alpha \leq 1-\frac{d}{q}$.

\medskip
Fix $x$ and $x'$ in $\mathbb{R}^d$ and set $r:=\abs{x-x'}$.
For the sake of simplicity, we define  $\rho^{\varepsilon}:=\rho_1^{\varepsilon}+\rho_2^{\varepsilon}$.
Let us prove that $\nabla S^{\varepsilon}$ is bounded in $L^{\infty}([0,\tau],C^{0,\alpha}(\mathbb{R}^d))$.
We split the integral into two parts $I_1$ and $I_2$.
\begin{equation*}
 \begin{aligned}
  \nabla S^{\varepsilon}(x,t)-\nabla S^{\varepsilon}(x',t)&=\int_0^t \int_{\mathbb{R}^d} \big(\nabla K(x-y,t-s)-\nabla K(x'-y,t-s)\big)(\rho_1^{\varepsilon}+\rho_2^{\varepsilon})(y,s) \,dy\,ds\\
  \phantom{\nabla S^{\varepsilon}(x,t)-\nabla S^{\varepsilon}(x',t)}& =I_1+I_2,
 \end{aligned}
\end{equation*}
with
\begin{equation*}
 \begin{aligned}
  I_1:&=\int_0^t \int_{B_x(2r)}\big(\nabla K(x-y,t-s)-\nabla K(x'-y,t-s)\big)(\rho_1^{\varepsilon}+\rho_2^{\varepsilon})(y,s)\,dy\,ds,\\
  I_2:&=\int_0^t \int_{\mathbb{R}^d \setminus B_x(2r)}\big(\nabla K(x-y,t-s)-\nabla K(x'-y,t-s)\big)(\rho_1^{\varepsilon}+\rho_2^{\varepsilon})(y,s)\,dy\,ds.
 \end{aligned}
\end{equation*}
 We now estimate $I_1$.
Using the definition of $\nabla K$ and the triangle inequality, we get
\begin{equation*}
 \begin{aligned}
   I_1\leq & \frac{1}{2(4\pi)^{d/2}}\int_0^t\!\!\!\int_{B_x(2r)}e^{-\frac{\abs{x-y}^2}{4(t-s)}-(t-s)}\frac{\abs{x-y}}{(t-s)^{d/2+1}}\rho^{\varepsilon}(y,s)\,dy\,ds \\
  \phantom{I_1= } & + \frac{1}{2(4\pi)^{d/2}} \int_0^t\!\!\!\int_{B_x(2r)} e^{-\frac{\abs{x'-y}^2}{4(t-s)}-(t-s)}\frac{\abs{x'-y}}{(t-s)^{d/2+1}} \rho^{\varepsilon}(y,s)\,dy\,ds.
 \end{aligned}
\end{equation*}
These two terms are identical and we will be treated in the same manner. 
We just carry out computations for the first one. 
Since $y$ belongs to $B_x(2r)$, $r=|x-x'|$, we have 
\begin{multline*}
 \int_0^t\!\!\!\int_{B_x(2r)}\!\!\!e^{-\frac{\abs{x-y}^2}{4(t-s)}-(t-s)}\frac{\abs{x-y}}{(t-s)^{d/2+1}}\rho^{\varepsilon}(y,s)\,dy\,ds \leq 2^{\alpha} \abs{x-x'}^{\alpha}\\
 \times \int_0^t\!\!\!\int_{B_x(2r)}e^{-\frac{\abs{x-y}^2}{4(t-s)}-(t-s)}\frac{\abs{x-y}^{1-\alpha}}{(t-s)^{d/2+1}}\rho^{\varepsilon}(y,s)\,dy\,ds.
\end{multline*}
Therefore, $I_1$ reads
\begin{multline*}
 I_1\leq \,\frac{2^{\alpha-1}}{(4\pi)^{d/2}} \abs{x-x'}^{\alpha} \left(\int_0^t\!\!\!\int_{B_x(2r)}\!\!e^{-\frac{\abs{x-y}^2}{4(t-s)}-(t-s)}\frac{\abs{x-y}^{1-\alpha}}{(t-s)^{d/2+1}}\rho^{\varepsilon}(y,s)\,dy\,ds \right.\\
\left. +\int_0^t\!\!\!\int_{B_x(2r)}\!\!e^{-\frac{\abs{x'-y}^2}{4(t-s)}-(t-s)} \frac{\abs{x'-y}^{1-\alpha}}{(t-s)^{d/2+1}}\rho^{\varepsilon}(y,s)\,dy\,ds \right).
\end{multline*}
We conclude by bounding the integral over the ball $B_x(2r)$ by the integral over the whole space.
We now deal with $I_2$. We first compute the hessian $D^2_x K$ of the kernel $K$.
\begin{equation*}
 D^2_x K(x,t)=\frac{1}{(4\pi t)^{d/2}}(-\frac{I_d}{t}+\frac{x\otimes x}{t^2})e^{-\frac{\abs{x}^2}{4t}-t}.
\end{equation*}
We decompose $D^2_x K$ into two parts~:
\begin{equation*}
 D^2_x K(x,t)=\frac{1}{(4\pi)^{d/2}}\left(-I_d+\frac{x}{\sqrt{t}}\otimes \frac{x}{\sqrt{t}}\right)e^{-\frac{1}{8} \abs{\frac{x}{\sqrt{t}}}^2} \times \frac{e^{-\frac{\abs{x}^2}{8t}-t}}{t^{d/2+1}}.
\end{equation*}
By using the change of variable $u=\frac{x}{\sqrt{t}}$, we notice that the first part of the rhs 
is bounded by a nonnegative constant $C$. Thus, 
\begin{equation*}
 \abs{D^2_x K}(x,t) \leq \frac{C}{t^{d/2+1}} e^{-\frac{\abs{x}^2}{8t}-t}.
\end{equation*}
Moreover, we remark that for all $y$ in $\mathbb{R}^d \setminus B_x(2r)$ and $u$ in $[0,1]$
\begin{equation*}
 \abs{x-y+u(x'-x)}\geq \abs{x-y}-u\abs{x-x'} \geq (1-u/2)\abs{x-y}.
\end{equation*}
We infer that for all $y$ in $\mathbb{R}^d \setminus B_x(2r)$ and $u$ in $[0,1]$
\begin{equation*}
 \abs{D^2_x K}(x-y+u(x'-x),t-s) \leq \frac{C}{(t-s)^{d/2+1}} \exp \left(-\frac{\abs{x-y}^2}{32(t-s)}-(t-s)\right).
\end{equation*}
Hence, 
\begin{equation*}
 I_2 \leq  C \abs{x-x'} \int_0^t\!\!\!\int_{\mathbb{R}^d\setminus B_x(2r) } \frac{1}{(t-s)^{d/2+1}}\exp\left(-\frac{\abs{x-y}^2}{32(t-s)}-(t-s)\right)\rho^{\varepsilon}(y,s) \,dy\,ds.
\end{equation*}
We deduce that 
\begin{equation*}
\begin{aligned}
 &I_2 \leq C'\abs{x-x'}^{\alpha} \int_0^t\!\!\!\int_{\mathbb{R}^d \setminus B_x(2r) } \frac{\abs{x-y}^{1-\alpha}}{(t-s)^{d/2+1}} \exp\left(-\frac{\abs{x-y}^2}{32(t-s)}-(t-s)\right) \rho^{\varepsilon}(y,s) \,dy\,ds,\\
 &\phantom{I_2}\leq C'\abs{x-x'}^{\alpha} \int_0^t\!\!\!\int_{\mathbb{R}^d} \frac{\abs{x-y}^{1-\alpha}}{(t-s)^{d/2+1}} \exp\left(-\frac{\abs{x-y}^2}{32(t-s)}-(t-s)\right)\rho^{\varepsilon}(y,s) \,dy\,ds.
\end{aligned}
\end{equation*}
Then using Young's inequality, we have for any positive constant $c$  and any $p\in (d,q)$,
\begin{equation*}
  \int_0^t \int_{\mathbb{R}^d} \frac{\abs{x-y}^{1-\alpha}}{(t-s)^{d/2+1}} e^ {-\frac{c\abs{x-y}^2}{t-s}-(t-s)} \rho^{\varepsilon}(y,s) dy ds \leq 
  \int_0^t \frac{e^{-(t-s)}}{(t-s)^{d/2+1}}\|\abs{y}^{1-\alpha} e^{-\frac{c\abs{y}^2}{t-s}}\|_{L^{p'}(\mathbb{R}^d)} \|\rho^{\varepsilon}\|_{L^{p}(\mathbb{R}^d)} ds.
\end{equation*}
It is clear that
\begin{equation*}
 \|y\mapsto \abs{y}^{1-\alpha} \exp\Big(-\frac{c\abs{y}^2}{t-s}\big)\|_{L^{p'}(\mathbb{R}^d)} \leq C(t-s)^{\frac{1-\alpha}{2}+\frac{d}{2p'}}.
\end{equation*}
This implies that for all $\alpha <1-\frac{d}{q}$ and all $d<p<q$
\begin{equation*}
 \int_0^t \int_{\mathbb{R}^d} \frac{\abs{x-y}^{1-\alpha}}{(t-s)^{d/2+1}} e ^{-\frac{c\abs{x-y}^2}{t-s}-(t-s)} \rho^{\varepsilon}(y,s)\,dy\,ds 
\leq C \left(\int_0^{\tau} t^{\frac{1-\alpha}{2}-\frac{d}{2p'}}e^{-t} dt \right)\|\rho^{\varepsilon}\|_{L^{\infty}([0,\tau],L^{p}(\mathbb{R}^d))}<\infty.
\end{equation*}
So, the $C^{1,\alpha}$ bound of $S^{\varepsilon}$ follows.
\end{proof}

\begin{lemma}\label{ineq}
Fix $S\geq 0$ and let $f$ be a velocity distribution and $q>1$.
Then, we have  
\begin{multline}\label{eq_min_sq}
\abs{\phi_i^{A,\varepsilon}[S](f+f')(f^{q-1}-(f')^{q-1})}\leq \frac{1}{2} \phi_i^{S,\varepsilon}[S](f-f')(f^{q-1}-(f')^{q-1})\\
+ \frac{\phi_i^{A,\varepsilon}[S]^2}{2\phi_i^{S,\varepsilon}[S^{\varepsilon}]} \frac{(f+f')^2 (f^{q-1}-(f')^{q-1})}{f-f'},
 \end{multline}
where $\phi_i^{S,\varepsilon}[S],\phi_i^{A,\varepsilon}[S]$ are the symmetric and anti-symmetric parts of $\mathcal{T}_i^{\varepsilon}[S]$ defined in \eqref{decomp_T}.
There also exists a constant $c_q$ depending on $q$ such that
\begin{equation}\label{eq_conv}
  \frac{(f+f')^2(f^{q-1}-(f')^{q-1})}{f-f'}\leq c_q(f^{q}+(f')^{q}).
\end{equation}
\end{lemma}

\begin{proof}
For \eqref{eq_min_sq}, we observe that
 \begin{equation*}
  \left(\phi_i^{S,\varepsilon}[S](f-f')\pm \phi_i^{A,\varepsilon}[S](f+f')\right)^2\geq 0.
 \end{equation*}
Expanding this inequality and dividing both sides by $\displaystyle (f^{q-1}-(f')^{q-1})/(f-f')$ gives the result.\\
To prove \eqref{eq_conv}, we show that for $a\geq 0$, the following function $A$ is bounded
\begin{equation*}
 A: a\longmapsto \frac{(1+a)^2(1-a^{q-1})}{(1-a)(1+a^q)}.
\end{equation*}
In fact, we have $\lim_{a\to 1} A(a)=2(q-1)>0$. 
Then $A$ can be continuously extended to $[0,\infty[$. 
Since $A$ tends to $1$ at infinity and is positive, 
$A$ is bounded by a positive constant.
\end{proof}
\begin{proposition}\label{uniform_estimate}
Assume that the initial data $f_1^{ini}$ and $f_2^{ini}$ belong to $L^1_{+}(\Omega) \cap L^{\infty}(\Omega)$.
Under Assumption (H1), the solution of \eqref{kinetics_eps}--\eqref{eq_S_eps} admits uniform estimates in $\varepsilon$ in the following spaces~:
\begin{equation*}
 \left\{
 \begin{aligned}
  & f_1^{\varepsilon}, f_2^{\varepsilon} \in L^{\infty}_{loc}((0,\infty),L^q(\Omega)),\quad 1 \leq q<\infty,\\
  & r_1^{\varepsilon}, r_2^{\varepsilon} \in L^2(\Omega \times (0,\infty)),\\
  & S^{\varepsilon} \in L^{\infty}_{loc}((0,\infty),C^{1,\alpha}(\mathbb{R}^d) \cap W^{1,p}(\mathbb{R}^d)),\quad 1\leq p \leq \infty,\, 0<\alpha<1,
 \end{aligned}
\right.
\end{equation*}
where $r_i^{\varepsilon}$ is given by 
\begin{equation*}
r_i^{\varepsilon}:=\frac{f_i^{\varepsilon}-\rho_i^{\varepsilon}F}{\varepsilon}.
\end{equation*}
\end{proposition}
The proof is the same in parabolic and elliptic cases and we carry out the proof only in the parabolic setting.

\begin{proof}
Let $\varepsilon>0$, $\tau>0$, $1\leq q <\infty$, $1<p\leq \infty$ and $0<\alpha<1$. The existence of $f_1^{\varepsilon},f_2^{\varepsilon},S^{\varepsilon}$ follows from Theorem~\ref{existence}. We remind the general conservation 
law for $\rho_i^{\varepsilon}$
\begin{equation}\label{cell_conservation_eps}
 \partial_t \rho_i^{\varepsilon}+\nabla \cdot J_i^{\varepsilon}=0,
\end{equation}
with $\displaystyle J_i^{\varepsilon}:=\frac{1}{\varepsilon} \int_V vf_i^{\varepsilon} dv$.
This gives that $f_i^{\varepsilon}(\cdot,t) \in L^1(\Omega)$ and we have
  \begin{equation*}
  \|f_i^{\varepsilon}(\cdot,t)\|_{L^1(\Omega)}=  \|f_i^{ini}\|_{L^1(\Omega)} .
 \end{equation*}
We suppose that $q>1$ and now prove the uniform boundness of $f_i^{\varepsilon}(\cdot,\tau)$ in $L^q(\Omega)$.
Multiplying equation for $f_i^{\varepsilon}$ \eqref{kinetics_eps} by $(f_i^{\varepsilon})^{q-1}$ and  integrating over $V$ and $\mathbb{R}^d$ gives 
\begin{equation}\label{eq_evfp}
  \frac{1}{q}\frac{d}{dt}\int_{\mathbb{R}^d}\!\!\int_V (f_i^{\varepsilon})^q \,dv\,dx=-\frac{1}{\varepsilon^2}\int_{\mathbb{R}^d}\!\!\int_V \mathcal{T}_i^{\varepsilon}[S^{\varepsilon}](f_i^{\varepsilon}) (f_i^{\varepsilon})^{q-1}\,dv\,dx.
\end{equation}
Let $\chi$ and $g$ be two real-valued functions. From the decomposition of $\mathcal{T}_i^{\varepsilon}[S]$ in its symmetric and anti-symmetric parts, we have
\begin{multline*}
 \int_V \mathcal{T}_i^{\varepsilon}[S](g)\chi(g)\,dv=\frac{1}{2}\int_V\!\!\int_V \phi_i^{S,\varepsilon}[S](g-g')(\chi(g)-\chi(g'))\,dv'\,dv\\
  +\frac{1}{2}\int_V\!\!\int_V \phi_i^{A,\varepsilon}[S](g+g')(\chi(g)-\chi(g'))\,dv'\,dv.
\end{multline*}
When applied to $g=f_i^{\varepsilon}$ and $\chi=x^{q-1}$, we get from \eqref{eq_evfp}
\begin{multline}\label{entropy_ineq}
 \frac{1}{q}\frac{d}{dt}\int_{\mathbb{R}^d}\!\!\int_V (f_i^{\varepsilon})^q \,dv\,dx+\frac{1}{2\varepsilon^2}\int_{\mathbb{R}^d}\!\!\int_V\!\!\int_V \phi_i^{S,\varepsilon}[S^{\varepsilon}](f_i^{\varepsilon}-(f_i^{\varepsilon})^{'})((f_i^{\varepsilon})^{q-1}-((f_i^{\varepsilon})^{'})^{q-1})\,dv'\,dv\,dx\\
  =-\frac{1}{2\varepsilon^2}\int_{\mathbb{R}^d}\!\!\int_V\!\!\int_V \phi_i^{A,\varepsilon}[S^{\varepsilon}](f_i^{\varepsilon}+(f_i^{\varepsilon})^{'})((f_i^{\varepsilon})^{q-1}-((f_i^{\varepsilon})^{'})^{q-1})\,dv'\,dv\,dx.
\end{multline}
We remark that the term depending on $\phi_i^{S,\varepsilon}[S]$ is positive whereas the sign of the other term is unknown. However, Lemma~\ref{ineq} allows us to bound it by two terms whose signs are known.

In fact, combining \eqref{eq_min_sq} and \eqref{eq_conv} gives
\begin{multline*}
  \abs{\phi_i^{A,\varepsilon}[S^{\varepsilon}](f_i^{\varepsilon}+(f_i^{\varepsilon})^{'})((f_i^{\varepsilon})^{q-1}-((f_i^{\varepsilon})^{'})^{q-1})} \leq \frac{1}{2} \phi_i^{S,\varepsilon}[S^{\varepsilon}](f_i^{\varepsilon}-(f_i^{\varepsilon})^{'})((f_i^{\varepsilon})^{q-1}-((f_i^{\varepsilon})^{'})^{q-1})\\
  + c_q\frac{\phi_i^{A,\varepsilon}[S^{\varepsilon}]^2}{2\phi_i^{S,\varepsilon}[S^{\varepsilon}]} ((f_i^{\varepsilon})^{q}+((f_i^{\varepsilon})^{'})^{q}).
\end{multline*}
Using this inequality in \eqref{entropy_ineq}, we get
 \begin{multline}\label{entropy_ineq_bis}
\frac{1}{q}\frac{d}{dt}\int_{\mathbb{R}^d}\!\!\int_V (f_i^{\varepsilon})^q \,dv\,dx+\frac{1}{4\varepsilon^2}\int_{\mathbb{R}^d}\!\!\int_V\!\!\int_V \phi_i^{S,\varepsilon}[S^{\varepsilon}](f_i^{\varepsilon}-(f_i^{\varepsilon})^{'})((f_i^{\varepsilon})^{q-1}-((f_i^{\varepsilon})^{'})^{q-1})\,dv'\,dv\,dx\\
  \leq\frac{c_q}{2\varepsilon^2}\int_{\mathbb{R}^d}\!\!\int_V\!\!\int_V \frac{\phi_i^{A,\varepsilon}[S^{\varepsilon}]^2}{2\phi_i^{S,\varepsilon}[S^{\varepsilon}]} ((f_i^{\varepsilon})^{q}+((f_i^{\varepsilon})^{'})^{q})\,dv\,dv'\,dx.
\end{multline}
We deduce that
\begin{equation*}
\frac{1}{q}\frac{d}{dt}\int_{\mathbb{R}^d}\!\!\int_V (f_i^{\varepsilon})^q\,dv\,dx \leq\frac{c_q}{2\varepsilon^2}\int_{\mathbb{R}^d}\!\!\int_V\!\!\int_V \frac{\phi_i^{A,\varepsilon}[S^{\varepsilon}]^2}{2\phi_i^{S,\varepsilon}[S^{\varepsilon}]} ((f_i^{\varepsilon})^{q}+((f_i^{\varepsilon})^{'})^{q})\,dv\,dv'\,dx.
\end{equation*}
Applying the Fubini theorem to the right-hand side and using \eqref{ineq_phi}, one gets 
\begin{equation}\label{estimate_f}
 \frac{d}{dt}\int_{\mathbb{R}^d}\!\!\int_V (f_i^{\varepsilon})^q\,dv\,dx \leq \frac{c_q q}{2}\psi_i \abs{V} \|\theta_i\|_{L^{\infty}(\mathbb{R})}^2 \int_{\mathbb{R}^d}\!\!\int_V (f_i^{\varepsilon})^q \,dv\,dx.
\end{equation}
Applying the Gronwall lemma, we conclude that for any $\tau>0$ and any $q\geq 1$,
$f_i^{\varepsilon}$ is uniformly bounded in $L^{\infty}([0,\tau],L^q(\Omega))$,
$i=1,2$.

We now prove the uniform boundness of $r_i^{\varepsilon}$ in $L^2\big(\Omega \times [0,\tau]\big)$.
Applying \eqref{entropy_ineq_bis} to $q=2$ and integrating over $t$ in $(0,\tau)$, we obtain
\begin{multline*}
 \int_0^{\tau}\!\!\int_{\mathbb{R}^d}\!\!\int_V\!\!\int_V \phi_i^{S,\varepsilon}[S^{\varepsilon}]\big(f_i^{\varepsilon}-(f_i^{\varepsilon})^{'}\big)^2 \,dv'\,dv \,dx \,dt \leq 2\varepsilon^2 \int_{\mathbb{R}^d}\!\!\int_V \left((f_i^{ini})^2 -(f_i^{\varepsilon})^2(x,v,\tau)\right) \,dv\,dx\\
  +2 c_2  \int_0^{\tau}\!\!\int_{\mathbb{R}^d}\!\!\int_V\!\!\int_V \frac{\phi_i^{A,\varepsilon}[S^{\varepsilon}]^2}{2\phi_i^{S,\varepsilon}[S^{\varepsilon}]} ((f_i^{\varepsilon})^{2}+((f_i^{\varepsilon})^{'})^{2}) \,dv\,dv'\,dx\,dt.
\end{multline*}
By the symmetry of $\displaystyle \frac{\phi_i^{A,\varepsilon}[S^{\varepsilon}]^2}{2\phi_i^{S,\varepsilon}[S^{\varepsilon}]}$, we have
\begin{multline*}
 \int_0^{\tau}\!\!\int_{\mathbb{R}^d}\!\!\int_V\!\!\int_V \phi_i^{S,\varepsilon}[S^{\varepsilon}](f_i^{\varepsilon}-(f_i^{\varepsilon})^{'})^2 \,dv'\,dv \,dx \,dt\leq 2\varepsilon^2\int_{\mathbb{R}^d}\!\!\int_V (f_i^{ini})^2 \,dv\,dx\\
 +2 c_2\int_0^{\tau}\!\!\int_{\mathbb{R}^d}\!\!\int_V\!\!\int_V \frac{\phi_i^{A,\varepsilon}[S^{\varepsilon}]^2}{\phi_i^{S,\varepsilon}[S^{\varepsilon}]} (f_i^{\varepsilon})^{2} \,dv\,dv'\,dx\,dt.
\end{multline*}
Since $f_i^{\varepsilon}$ is uniformly bounded in $L^{\infty}([0,\tau],L^2(\Omega))$, then from the Fubini Theorem and \eqref{ineq_phi} we have
\begin{equation*}
\int_0^{\tau}\!\!\int_{\mathbb{R}^d}\!\!\int_V\!\!\int_V \phi_i^{S,\varepsilon}[S^{\varepsilon}](f_i^{\varepsilon}-(f_i^{\varepsilon})^{'})^2 \,dv'\,dv\,dx\,dt \leq c\, \varepsilon^2.
\end{equation*}
From \eqref{ineq_phi}, we have
\begin{equation*}
\psi_i \int_0^{\tau}\!\!\int_{\mathbb{R}^d}\!\!\int_V\!\!\int_V (f_i^{\varepsilon}-(f_i^{\varepsilon})^{'})^2 \,dv'\,dv \,dx \,dt \leq c\,\varepsilon^2.
\end{equation*}
We rewrite this inequality in terms of $r_i^{\varepsilon}$ and obtain
\begin{equation*}
\int_0^{\tau}\!\!\int_{\mathbb{R}^d}\!\!\int_V\!\!\int_V \left(r_i^{\varepsilon}-(r_i^{\varepsilon})^{'}\right)^2 \,dv'\,dv \,dx \,dt\leq \frac{c}{\psi_i}.
\end{equation*}
Expanding the left-hand side and using $\displaystyle \int_V r_i^{\varepsilon}dv=0$ leads to
\begin{equation*}
 \int_0^{\tau}\!\!\!\int_{\mathbb{R}^d}\!\!\int_V (r_i^{\varepsilon})^2 \,dv\,dx\,dt
  \leq \frac{c}{2\psi_i \abs{V}}.
\end{equation*}
We finish with uniform estimates on $S^{\varepsilon}$. Applying Lemma ~\ref{estimate_S} for a $q>d$, it follows that
\begin{multline*}
  \|S^{\varepsilon}\|_{L^{\infty}([0,\tau],W^{1,p}(\mathbb{R}^d))}+\|S^{\varepsilon}\|_{L^{\infty}([0,\tau],C^{1,\alpha}(\mathbb{R}^d))}\leq \\
 c_\tau \big(\|\rho_1^{\varepsilon}\|_{L^{\infty}([0,\tau],L^1(\mathbb{R}^d))}+\|\rho_2^{\varepsilon}\|_{L^{\infty}([0,\tau],L^1(\mathbb{R}^d))}
 + \|\rho_1^{\varepsilon}\|_{L^{\infty}([0,\tau],L^q(\mathbb{R}^d))}+\|\rho_2^{\varepsilon}\|_{L^{\infty}([0,\tau],L^q(\mathbb{R}^d))}\big).
\end{multline*}
We know that $f_i^{\varepsilon}$ is uniformly bounded in $L^{\infty}([0,\tau],L^q(\Omega))$. Using the total mass conservation, we deduce
that $S^{\varepsilon}$ is uniformly bounded in $L^{\infty}([0,\tau],C^{1,\alpha}(\mathbb{R}^d)\cap W^{1,p}(\mathbb{R}^d))$ for any $p$ and $\alpha$ verifying $1\leq p \leq \infty$ and $0<\alpha<1$.
\end{proof}

\subsection{Proof of Theorem \ref{limit}}
We prove the theorem in the parabolic setting $\delta =1$. 
The proof in the elliptic one is analogous.

We first recall the following Aubin-Lions-Simon compactness lemma.
\begin{lemma}[\textbf{Aubin-Lions-Simon}, \cite{JS}]\label{Aubin_Lions}
Let $\tau>0$ and $p,q$ be such that $1<p\leq \infty,1\leq q \leq \infty$. Let $V,E,F$ be Banach spaces such that 
\begin{equation*}
 V\overset{c}\hookrightarrow E \hookrightarrow F.
\end{equation*}
If $A$ is bounded in $W^{1,p}((0,\tau),F) \cap L^q((0,\tau),V)$, then $A$ is relatively compact in $L^q((0,\tau),E)$.
\end{lemma}


{\it Proof of Theorem \ref{limit}.} Let us fix $1<q<\infty$, $1<p\leq \infty$ and $\tau>0$. 
We decomposed the proof into three steps.

{\bf Step 1~:} Uniform estimates for $\partial_t(\nabla S^{\varepsilon})$ and $\partial_t(S^{\varepsilon})$.
We differentiate the conservation law \eqref{cell_conservation_eps} of $\rho_i^{\varepsilon}$  with respect to $x$.
\begin{equation*}
  \partial_t \nabla_x \rho_i^{\varepsilon}+\nabla_x \nabla \cdot J_i^{\varepsilon}=0 .
\end{equation*}
Summing these equalities for $i=1,2$ gives 
\begin{equation*}
\partial_t \nabla_x (\rho_1^{\varepsilon}+\rho_2^{\varepsilon})+\nabla_x \nabla \cdot(J_1^{\varepsilon}+J_2^{\varepsilon})=0.
\end{equation*}
By multiplying by the kernel $K(t-s,x-y)$ and integrating by parts, one gets 
 \begin{equation}\label{eq_gradS}
  \partial_t (\nabla_x S^{\varepsilon})+\nabla_x(\nabla \cdot S^{J,\varepsilon})=\nabla_x \left(\int_{\mathbb{R}^d}\!\!K(t,y)(\rho_1+\rho_2)(0,x-y)\,dy\right),
 \end{equation}
where $S^{J,\varepsilon}$ denotes
\begin{equation*}
 S^{J,\varepsilon}:=\int_0^t\!\!\int_{\mathbb{R}^d}\!\!K(t-s,x-y)(J_1^{\varepsilon}+J_2^{\varepsilon})(s,y)\,dy\,ds.
\end{equation*}
From Proposition~\ref{uniform_estimate}, $r_i^{\varepsilon}$ is uniformly bounded in $L^2(\Omega \times [0,\tau])$.
Since $J_i^{\varepsilon}=\int_V vr_i^\eps\,dv$, we deduce by the Cauchy-Schwarz
inequality that $J_i^\eps$ is also bounded in $L^2(\Omega \times [0,\tau])$.

The mathematical form of $S^{J,\varepsilon}$ implies that $S^{J,\varepsilon}$ 
satisfies the same parabolic equation as $S^{\varepsilon}$ with the 
right-hand side $J_1^{\varepsilon}+J_2^{\varepsilon}$.
Using the parabolic regularity, we conclude that 
 \begin{equation*}
  S^{J,\varepsilon} \in L^2((0,\tau],H^2_{\text{loc}}(\mathbb{R}^d)).
 \end{equation*}
 Then, from \eqref{eq_gradS}, we get that
\begin{equation*}
 \partial_t (\nabla S^{\varepsilon}) \in L^2((0,\tau],L^2_{\text{loc}}(\mathbb{R}^d)).
\end{equation*}
Moreover, we know the expression of $\partial_t S^{\varepsilon}$ from Lemma~\ref{diff_phi}.
\begin{equation*}
\partial_t S^{\varepsilon} =\int_{\mathbb{R}^d}\!\!\!K(y,\tau)(\rho_1^{ini}+\rho_2^{ini})(x-y)\,dy+\int_0^{\tau}\!\!\!\int_{\mathbb{R}^d}\!\!\!\nabla_x K(y,s)(J_1^{\varepsilon}+J_2^{\varepsilon})(x-y,\tau-s)\,dy\,ds.
\end{equation*}
The same conclusion holds for $\partial_t S^{\varepsilon}$~:
\begin{equation*}
 \partial_t S^{\varepsilon} \in L^2((0,\tau],L^2_{\text{loc}}(\mathbb{R}^d)).
\end{equation*}

{\bf Step 2~:} Extraction of subsequences.
By the Ascoli Theorem, we have the following embeddings~:
\begin{equation*}
 C^{0,\alpha}_{\text{loc}}(\mathbb{R}^d) \overset{c} \hookrightarrow L^p_{\text{loc}}(\mathbb{R}^d) \hookrightarrow  L^2_{\text{loc}}(\mathbb{R}^d),\quad 2\leq p \leq \infty.
\end{equation*}
From the Aubin-Lions-Simon Lemma~\ref{Aubin_Lions}, 
there exists subsequences $(S^{\varepsilon})_\eps$ and $(\nabla S^{\varepsilon})_\eps$ that strongly converge in $L^p_{\text{loc}}(\mathbb{R}^d \times (0,\tau])$. 
This result is extended to $p$ between $1$ and $2$ by the continuous embedding of $L^{\infty}_{\text{loc}}$ into $L^p_{\text{loc}}$. 
The extraction of weak-* convergent subsequences of $f_1^{\varepsilon},f_2^{\varepsilon}$ and weak convergent subsequences $r_1^{\varepsilon},r_2^{\varepsilon}$ follows from the separability or reflexivity property of Banach spaces $L^{\infty}((0,\tau],L^q(\Omega))$, $L^2(\Omega \times (0,\tau])$.\medskip\\
We find a subsequence of $(f_1^{\varepsilon},f_2^{\varepsilon},S^{\varepsilon})$ which satisfies
\begin{equation*}
\begin{aligned}
 f_1^{\varepsilon}&\overset{\ast}{\rightharpoonup} f_1^0,\quad f_2^{\varepsilon}\overset{\ast}{\rightharpoonup} f_2^0\quad\text{in } L^{\infty}((0,\tau],L^q(\Omega)),\\   
 r_1^{\varepsilon}&\rightharpoonup r_1^0,\quad r_2^{\varepsilon}\rightharpoonup r_2^0,\quad \text{in }L^2(\Omega \times (0,\tau]),\\
 S^{\varepsilon}&\rightarrow S^0,\quad \nabla_x S^{\varepsilon} \rightarrow \nabla_x S^0 \quad \text{in } L^p_{\text{loc}}(\mathbb{R}^d \times (0,\tau] ), 
\end{aligned}
\end{equation*}
for any $p,q$ such that $1<q<\infty$ and $1\leq p \leq \infty$.\\
As a consequence, $\rho_i^{\varepsilon}$ converges weakly-* in $L^{\infty}((0,\tau],L^q(\mathbb{R}^d))$.
\begin{equation*}
 (\rho_1^{\varepsilon},\rho_2^{\varepsilon})\overset{\ast}{\rightharpoonup} \left(\int_V f_1^0 dv,\int_V f_2^0 dv \right) \quad\text{in } L^{\infty}((0,\tau_0),L^q(\mathbb{R}^d)).\\   
\end{equation*}

{\bf Step 3~:} Passing to the limit. 
This step consists in identifying limits $f_1^0,f_2^0,r_1^0,r_2^0$. We now operate in either $L^2_{\text{loc}}((0,\tau],L^2(\mathbb{R}^d))$ or $L^2_{\text{loc}}((0,\tau],L^2(\Omega))$. We pass to the limit in the relation $f_i^{\varepsilon}=\rho_i^{\varepsilon} F+\varepsilon r_i^{\varepsilon}$ to get
\begin{equation*}
 f_1^0=\rho_1^0 F\quad \text{and } \quad f_2^0=\rho_2^0 F \quad \text{in } L^2_{\text{loc}}((0,\tau],L^2(\Omega)).
\end{equation*}
We replace $f_i^{\varepsilon}$ by $\rho_i^{\varepsilon} F+\varepsilon r_i^{\varepsilon}$ into $\mathcal{T}_i^{0}(f_i^{\varepsilon})$, the equation for $f_i^{\varepsilon}$ \eqref{kinetics_eps} now reads
\begin{equation}\label{eq_f_eps}
 \varepsilon \partial_t f_i^{\varepsilon}+ v\cdot \nabla_x f_i^{\varepsilon} =-\abs{V}\psi_i r_i^{\varepsilon}+\psi_i \left(\int_V \theta_i(\varepsilon \partial_t S^{\varepsilon}+v'\cdot \nabla_x S^{\varepsilon}){f_i^{\varepsilon}}^{'}dv'-\abs{V} \theta_i(\varepsilon S^{\varepsilon}+v\cdot \nabla_x S^{\varepsilon})f_i^{\varepsilon}\right).
\end{equation}
From the previous study, we have 
\begin{center}
 $\partial_t S^{\varepsilon}$ is uniformly bounded in $L^2((0,\tau],L^2_{\text{loc}}(\mathbb{R}^d)),$  
\end{center}
and 
\begin{center}
$\nabla_x S^{\varepsilon} \rightarrow \nabla_x S^0$ in $L^2((0,\tau],L^2_{\text{loc}}(\mathbb{R}^d))$.
\end{center}
Since $\theta_i$ is Lipschitz continuous, we deduce that 
\begin{equation*}
 \theta_i(\varepsilon \partial_t S^{\varepsilon}+v\cdot\nabla_x S^{\varepsilon})\rightarrow \theta_i(v\cdot\nabla_x S^0) \;\text{in}\; L^2((0,\tau],L^2_{\text{loc}}(\Omega)).
\end{equation*}
We also know that 
\begin{equation*}
 f_i^{\varepsilon} \rightharpoonup \rho_i^0 F \text{ in } L^2((0,\tau],L^2_{\text{loc}}(\Omega)).
\end{equation*}
The combination of these arguments gives 
\begin{equation*}
 \begin{aligned}
  \int_V \theta_i(\varepsilon \partial_t S^{\varepsilon}+v'\cdot \nabla_x S^{\varepsilon}) {f_i^{\varepsilon} }^{'} dv' &\rightharpoonup  \rho_i^0 \int_V \theta_i(v'\cdot\nabla_x S^0)\,\frac{dv'}{\abs{V}}\quad \text{ in } L^2((0,\tau],L^2_{\text{loc}}(\Omega)),\\
   \theta_i(\varepsilon \partial_t S^{\varepsilon}+v\cdot \nabla_x S^{\varepsilon}) f_i^{\varepsilon} &\rightharpoonup \rho_i^0 \theta_i(v\cdot\nabla_x S^0)\,\frac{\mathds{1}_{v\in V}}{\abs{V}}\quad \text{ in } L^2((0,\tau],L^2_{\text{loc}}(\Omega)).
 \end{aligned}
\end{equation*}
By passing to the limit in \eqref{eq_f_eps}, we obtain $r_i^0$.
\begin{equation*}
r_i^0=-\frac{v\cdot \nabla_x \rho_i^0}{\psi_i \abs{V}^2} +\frac{\rho_i^0}{\abs{V}}\left(\int_V \theta_i(v'\cdot\nabla_x S^0)\frac{dv'}{\abs{V}}-\theta_i(v\cdot\nabla_x S^0)\right).
\end{equation*}
We deduce that $J_i^{\varepsilon}$ converges weakly in $L^2((0,\tau],L^2_{\text{loc}}(\mathbb{R}^d))$.
\begin{equation*}
 J_i^{\varepsilon}=\frac{1}{\varepsilon}\int_V v\,f_i^{\varepsilon}=\int_V v \,r_i^{\varepsilon} \rightharpoonup \int_V v r_i^0.
\end{equation*}
By passing to the limit in the conservation law \eqref{cell_conservation_eps}, we find the equation for $\rho_i^0$.
\begin{equation*}
 \partial_t \rho_i^{0}+\nabla \cdot \int_V v \,r_i^0 =0 .
\end{equation*}
Next, the equation for $S^0$ is obtained by passing to the limit in \eqref{eq_S_eps}.

\section*{Appendix~: Well-posedness of the macroscopic system }
In this appendix, we explain briefly the global well-posedness of solutions
to the parabolic/parabolic or parabolic/elliptic system \eqref{limit_equation}
provided the chemosensitivities $\chi_i[S]$ are bounded and $D_i$ are symmetric 
positive definite. More precisely, we consider the system
\begin{equation*}
 \left\{
 \begin{aligned}
  &\partial_t \rho_1=\nabla \cdot \left(D_1\nabla_x \rho_1-\chi_1[S]\rho_1\right),\\
  &\partial_t \rho_2=\nabla \cdot \left(D_2\nabla_x \rho_2-\chi_2[S]\rho_2\right),\\
  &\delta \partial_t S=\Delta S-S+\rho_1+\rho_2,\quad \delta=0,1,
 \end{aligned}
 \right.
\end{equation*}
complemented with the initial condition
\begin{equation*}
 \rho_{1}^{ini}=\int_V f_1^{ini} dv,\quad \rho_{2}^{ini}=\int_V f_2^{ini} dv,\quad   S^{ini}=0 \quad \mbox{ if } \delta=1.
\end{equation*}
Then we have the following a-priori estimate~:

\begin{proposition}
Let $t>0$ and $q\geq 2$. Let $(\rho_1,\rho_2,S)$ be a positive weak
solution of \eqref{limit_equation} such that $\rho_i^{ini} \in L^q(\mathbb{R}^d)$. 
Let us assume that there exists $\chi_i^\infty$ such that $\|\chi_i[S]\|_{L^\infty} \leq \chi_i^\infty$
and that there exists $D_i^{min}$ such that $D_i X\cdot X\geq D_i^{min}\|X\|^2$.
Then, there exists a constant $C>0$ such that for all $i=1,2$, and $\tau\in [0,t]$,
$$
\int_{\R^d} |\rho_i|^q\,dx \leq C \qquad 
\int_0^\tau \int_{\R^d} \rho_i^{q-2} |\nabla_x \rho_i|^2\,dx\,ds \leq C
$$
\end{proposition}

\begin{proof}
 We multiply the equation of $\rho_i$ by $\rho_i^{q-1}$ ($q\geq 2$) and integrate over $\mathbb{R}^d$.
 \begin{equation*}
  \frac{1}{q}\frac{d}{dt}\int_{\mathbb{R}^d} \rho_i^q \,dx =-\int_{\mathbb{R}^d}D_i \nabla_x \rho_i \cdot \nabla(\rho_i^{q-1})\,dx + \int_{\mathbb{R}^d}\rho_i \chi_i[S]\cdot \nabla_x (\rho_i^{q-1})\,dx. 
 \end{equation*}
Using assumptions on $D_i$ and $\chi_i[S]$ yields
\begin{equation*}
 \frac{1}{q}\frac{d}{dt}\int_{\mathbb{R}^d} \rho_i^q dx \leq-4D_i^{min}\frac{(q-1)}{q^2}\int_{\mathbb{R}^d} \abs{\nabla_x(\rho_i^{q/2})}^2 dx +2\chi_i^{\infty}\frac{q-1}{q}\int_{\mathbb{R}^d} \rho_i^{q/2}\abs{\nabla_x(\rho_i^{q/2})}dx,
\end{equation*}
Using the Cauchy-Schwarz inequality and Young's inequality $2ab\leq \eps a^2 + b^2/\eps$ with $\displaystyle \varepsilon=\frac{2D_i^{min}}{q \chi_i}$, we get
\begin{equation}\label{ineq_energy}
 \frac{1}{q}\frac{d}{dt}\int_{\mathbb{R}^d} \rho_i^q 
 + 2D_i^{min}\frac{(q-1)}{q^2}\int_{\mathbb{R}^d} \abs{\nabla_x(\rho_i^{q/2})}^2 dx
\leq (\chi_i^{\infty})^2\frac{q-1}{ 2D_i^{min}}\int_{\mathbb{R}^d} \rho_i^q dx.
\end{equation}
We conclude by the Gronwall inequality.
\end{proof}

In the case at hand, $D_i,\chi_i[S]$ are defined by
\begin{equation*}
 D_i=\frac{1}{\abs{V}^2 \psi_i}\int_V v \otimes v \,dv \quad \mbox{and } \chi_i[S]=-\int_V v \theta_i(v\cdot \nabla S)\,\frac{dv}{\abs{V}}.
\end{equation*}
Then each $D_i$ is a positive, defnite and diagonal matrix, since the domain $V$ is symmetric; $\chi_i[S]$ is also bounded. 
We can apply the previous result which implies the non blow-up in finite time
of weak solution. Then global existence is obtained by extending local existence
thanks to the above a priori estimate. 
We refer the interested reader to \cite{Chertock}, where the single-species 
case has been considered.

\bigskip

{\bf Acknowledgement. } The authors ackowledge partial support from the ANR project Kibord, 
ANR-13-BS01-0004 funded by the French Ministry of Research.


\bibliographystyle{plain}
\bibliography{paper_references}

\end{document}